\pdfoutput=1
\RequirePackage[l2tabu, orthodox]{nag}
\documentclass{amsart}

\usepackage[T1]{fontenc}
\usepackage[utf8]{inputenc}

      \title[$K$-theory, assembly maps, controlled algebra, and trace methods]%
            {Algebraic $\boldsymbol{K}$-theory, assembly maps,\\
             controlled algebra, and trace methods\\[\smallskipamount]
             \textmd{\small A primer and a survey of the Farrell-Jones Conjecture}}

     \author{Holger Reich}
    \address{Institut für Mathematik, Freie Universität Berlin, Germany}
      \email{\hemail{holger.reich@fu-berlin.de}}
    \urladdr{\hurl{mi.fu-berlin.de/math/groups/top/members/Professoren/reich.html}}

     \author{Marco Varisco}
    \address{Department of Mathematics and Statistics, University at Albany, SUNY, USA}
      \email{\hemail{mvarisco@albany.edu}}
    \urladdr{\hurl{albany.edu/~mv312143/}}

  \subjclass[2010]{\MSC{19-02}}
   \keywords{Algebraic $K$-theory, Farrell-Jones Conjecture}

       \date{February 23, 2018}

\usepackage{amssymb}
\usepackage{enumitem}
\usepackage{etoolbox}
\usepackage{mathtools}
  \mathtoolsset{showonlyrefs}
\usepackage{adjustbox}
\usepackage{pifont}
\usepackage{tikz}
  \usetikzlibrary{cd}
\usepackage{microtype}
\usepackage[noautomatic]{imakeidx}
  \makeindex[columns=1]
\usepackage[pagebackref, pdfinfo={
      Title={Algebraic K-theory, assembly maps, controlled algebra, and trace methods},
    Subject={A primer and a survey of the Farrell-Jones Conjecture},
     Author={Holger Reich and Marco Varisco},
   Keywords={Algebraic K-theory, Farrell-Jones Conjecture},
}]{hyperref}

\makeatletter
\def\@seccntformat#1{%
  \protect\textup{%
    \protect\@secnumfont
    \expandafter\protect\csname format#1\endcsname
    \csname the#1\endcsname
    \protect\@secnumpunct
  }%
}
\makeatother

\setlist{leftmargin=*}

\newcommand*{\hurl}  [2][www.]{\href{https://#1#2}{\nolinkurl{#2}}}
\newcommand*{\hemail}[1]{\href{mailto:#1}{\nolinkurl{#1}}}
\newcommand*{\hDOI}  [1]{\href{https://dx.doi.org/#1}{DOI~\nolinkurl{#1}}}
\newcommand*{\arXiv} [1]{\href{https://www.arxiv.org/abs/#1}{\nolinkurl{arXiv:#1}}}
\newcommand*{\MSC}   [1]{\href{https://www.ams.org/msc/msc2010.html?t=#1}{#1}}

\theoremstyle{plain}
\newtheorem{claim}[equation]{Claim}
\newtheorem{conjecture}[equation]{Conjecture}
\newtheorem{conjecture-FJ}[equation]{Farrell-Jones Conjecture}
\newtheorem{corollary}[equation]{Corollary}
\newtheorem{lemma}[equation]{Lemma}
\newtheorem{proposition}[equation]{Proposition}
\newtheorem{theorem}[equation]{Theorem}
\newtheorem{criterion}[equation]{Criterion}
\newtheorem{transitivity}[equation]{Transitivity Principle}

\theoremstyle{definition}
\newtheorem{example}[equation]{Example}
\newtheorem{examples}[equation]{Examples}
\newtheorem{remark}[equation]{Remark}
\newtheorem{definition}[equation]{Definition}

\newcommand*{\one}  {\text{\ding{192}}} \newcommand*{\bone}  {\text{\ding{202}}}
\newcommand*{\two}  {\text{\ding{193}}} \newcommand*{\btwo}  {\text{\ding{203}}}
\newcommand*{\three}{\text{\ding{194}}} \newcommand*{\bthree}{\text{\ding{204}}}
\newcommand*{\four} {\text{\ding{195}}} \newcommand*{\bfour} {\text{\ding{205}}}
\newcommand*{\five} {\text{\ding{196}}} 
\newcommand*{\six}  {\text{\ding{197}}} 
\newcommand*{\seven}{\text{\ding{198}}} 
\newcommand*{\eight}{\text{\ding{199}}} 
\newcommand*{\nine} {\text{\ding{200}}}

\DeclareMathAlphabet{\matheurm}{U}{eur}{m}{n}

\newcommand*{\define}[5]{%
  \ifstrequal{#2}{*}{\expandafter#1\expandafter*}{\expandafter#1}%
  \csname#4#5\endcsname{#3{#5}}
}
\forcsvlist{\define{\newcommand}{*}{\mathbf}{B}}{%
  A,B,C,D,E,F,G,H,I,J,K,L,M,N,O,P,Q,R,S,T,U,V,W,X,Y,Z,%
}
\forcsvlist{\define{\newcommand}{*}{\mathcal}{C}}{%
  A,B,C,D,E,F,G,H,I,J,K,L,M,N,O,P,Q,R,S,T,U,V,W,X,Y,Z,%
}
\forcsvlist{\define{\newcommand}{*}{\mathbb}{I}}{%
  A,B,C,D,E,F,G,H,I,  K,L,M,N,O,P,Q,R,S,T,U,V,W,X,Y,Z%
}                

\forcsvlist{\define{\DeclareMathOperator}{}{}{}}%
{asbl, coker, conhom, conj, hoeq, hofib, id, im, incl, ind, inn, Lan, map, mor, obj, pr, pt, res, size, supp, tr, trb, trc, trd}

\forcsvlist{\define{\DeclareMathOperator}{*}{}{}}%
{colim, hocolim, holim}

\forcsvlist{\define{\newcommand}{*}{\matheurm}{}}%
{Ab, Groupoids, Or, Sets, Sp, Sub, Top}

\forcsvlist{\define{\newcommand}{*}{\mathbf}{}}%
{K, Wh, NK, HH, THH, TC, C, T}
\newcommand*{\wh}{\mathit{Wh}}
\newcommand*{\hh}{\mathit{HH}}

\newcommand*{\TO}  [1][]{\stackrel{#1}{\longrightarrow}}

\newcommand*{\MOR} [4][]{#2\colon#3\TO[#1]#4}
\newcommand*{\AND}{\qquad\text{and}\qquad}
\newcommand*{\aND} {\quad\text{and}\quad}

\DeclarePairedDelimiterX\SET[2]{\{}{\}}{\,#1\;\delimsize\vert\;#2\,}

\newcommand*{\ds}{\displaystyle}
\newcommand*{\ts}{\textstyle}

\newcommand*{\op}{{\operatorname{op}}}
\newcommand*{\Eucl}{{\operatorname{Eucl}}}
\newcommand*{\compl}{{\wedge}}

\DeclareMathOperator*{\smallcoprod}{\ts\coprod}
\DeclareMathOperator*{\smallprod}  {\ts\prod}
\DeclareMathOperator*{\tensor}     {\otimes}
\DeclareMathOperator*{\timesd}     {\times}
\DeclareMathOperator*{\sma}        {\wedge}

\newcommand*{\Cyc}   {{\mathcal{C} \hspace{-.2ex}\mathit{yc}}}
\newcommand*{\VCyc}  {{\mathcal{VC}\hspace{-.2ex}\mathit{yc}}}
\newcommand*{\FCyc}  {{\mathcal{FC}\hspace{-.2ex}\mathit{yc}}}
\newcommand*{\MaxCyc}{{\mathcal{M} \hspace{-.2ex}\mathit{ax}\mspace{1.5mu}\Cyc}}
\newcommand*{\Fin}   {{\mathcal{F} \hspace{-.2ex}\mathit{in}}}

\newcommand*{\oid}[2]{#1\!\smallint\!#2} 

\newcommand*{\assum}[2]{$[#1_#2]$}

\newcommand*{\Cp}[2][p]{C_{{#1}^{#2}}}

\newcommand*{\FJclass}{\mathfrak{G}}

\newcommand*{\BHM}{Bökstedt-Hsiang-Madsen}

\newcommand*{\totient}{\varphi}

\begin{document}

\begin{abstract}
We give a concise introduction to the Farrell-Jones Conjecture in algebraic $K$-theory and to some of its applications.
We survey the current status of the conjecture, and we illustrate the two main tools that are used to attack it: controlled algebra and trace methods.
\end{abstract}


\maketitle
\microtypesetup{protrusion=false}
\tableofcontents
\microtypesetup{protrusion=true}
\thispagestyle{empty}


\section{Introduction}

The classification of manifolds and the study of their automorphisms are central problems in mathematics.
For manifolds of sufficiently high dimension, these problems can often be successfully solved using algebraic topological invariants in the algebraic $K$-theory and $L$-theory of group rings.

In an article published in 1993~\cite{FJ-iso}, Tom Farrell and Lowell Jones formulated a series of \emph{Isomorphism Conjectures} about the $K$ and~$L$-theory of group rings, which became universally known as the \emph{Farrell-Jones Conjectures}.
On the one hand these conjectures represented the culmination of decades of seminal work by Farrell, Jones, and Wu Chung Hsiang, e.g.~\cite{FH-space-form}, \cite{FH-Novikov}, \cite{FH-poly}, \cite{Hsiang-ICM}, \cite{FJ-K-dynamics}, \cite{FJ-Mostow}.
On the other hand they have motivated and continue to motivate an impressive body of research.

In this article we focus only on the Farrell-Jones Conjecture for algebraic $K$-theory, and mention briefly some of its variants in Subsection~\ref{subsec:related-conjectures}.
We give a concise introduction to this conjecture and to some of its applications, survey its current status, and most importantly we explain the main ideas and tools that are used to attack  the conjecture: controlled algebra and trace methods.

Section~\ref{sec:conjectures} begins with some fundamental conjectures in algebra and geometric topology, which can be reformulated in terms of $K_0$ and~$K_1$ of group rings.
These conjectures are all implied by the Farrell-Jones Conjecture, but they are more accessible and elementary; moreover, their importance and appeal do not require algebraic $K$-theory, but may serve as motivation to study it.

In Subsections~\ref{subsec:assembly} and~\ref{subsec:FJ} we define assembly maps and use them to formulate the Farrell-Jones Conjecture.
Then we discuss how the Farrell-Jones Conjecture implies all other conjectures discussed in this article.

In Section~\ref{sec:state} we collect most of what is known today about the Farrell-Jones Conjecture in algebraic $K$-theory.
We invite the reader to compare that section to the corresponding Section~2.6 in the survey article~\cite{LR-survey} from~2005, to appreciate the tremendous amount of activity and progress that has taken place since then.

The last two sections focus on proofs.
In Section~\ref{sec:controlled} we introduce the basic concepts of controlled algebra and see them at work.
In particular, we give an almost complete proof of the Farrell-Jones Conjecture in the simplest nontrivial case, that of the free abelian group on two generators.
Many ingenious ideas, mainly going back to Farrell and Hsiang, enter the proof already in this seemingly basic case.
This section is meant to be an accessible introduction to controlled algebra.
We do not even mention the very important flow techniques, and highly recommend Arthur Bartels’s survey article~\cite{Bartels-proofs}.

In Section~\ref{sec:trace} we illustrate how trace methods are used to prove rational injectivity results about assembly maps.
We give a complete proof of an elementary but illuminating statement about~$K_0$ in Subsection~\ref{subsec:warm-up}, and then explain how this idea can be generalized using more sophisticated tools like topological Hochschild homology and topological cyclic homology.
The complicated technical details underlying the construction of these tools are beyond the scope of this article, and we refer the reader to~\cite{DGM-book}, \cite{Hesselholt-handbook}, and~\cite{Madsen} for more information.
However, we carefully explain the structure of the proof of the algebraic $K$-theory Novikov Conjecture due to Marcel Bökstedt, Hsiang, and Ib Madsen~\cite{BHM}.
We follow the point of view used by the authors in joint work with Wolfgang Lück and John Rognes~\cite{kc}, leading to a generalization of this theorem for the Farrell-Jones assembly map.
In particular, we highlight the importance of a variant of topological cyclic homology, \BHM's functor~$C$, which has seemingly disappeared from the literature since~\cite{BHM}.

We tried to make our exposition accessible to nonexperts, and no deeper knowledge of algebraic $K$-theory is required.
However, we expect our reader to have seen the basic definitions and properties of $K_0$ and $K_1$, and to be willing to accept the existence of a spectrum-valued algebraic $K$-theory functor.
Classical and less classical sources for the $K$-theoretic background include
\cite{Bass-book},
\cite{Cortinas-friendly},
\cite{DGM-book},
\cite{Milnor-book},
\cite{Rosenberg},
and~\cite{Weibel}.

There are other survey articles about the Farrell-Jones and related conjectures: \cite{Bartels-proofs}, \cite{LR-survey},  and~\cite{Madsen}, which we already recommended, and also~\cite{L-ICM} and the voluminous book project~\cite{L-book-project}.
Our hope is that this contribution may serve as a more concise and accessible starting point, preparing the reader for these other more advanced surveys and for the original articles.

\subsubsection*{Acknowledgments}
This work was supported by the Collaborative Research Center~647 \emph{Space -- Time -- Matter} in Berlin and by a grant from the Simons Foundation (\#419561, Marco Varisco).


\section{Conjectures} \label{sec:conjectures}

In this section we discuss many conjectures related to group rings and their algebraic $K$-theory.
These conjectures are all implied by the Farrell-Jones Conjecture, which we formulate in Subsection~\ref{subsec:FJ}.
All of these conjectures are known in many cases but open in general, as we review in Section~\ref{sec:state}.


\subsection{Idempotents and projective modules}

An element $p$ in a ring is an idempotent if~$p^2=p$.
The trivial examples are the elements $0$ and~$1$.

\begin{conjecture}[trivial idempotents]
\label{conj:Kaplansky}
\index{conjecture!trivial idempotents}
Let $\Bbbk$ be a field of characteristic zero and let $G$ be a torsion-free group.
Then every idempotent in the group ring $\Bbbk[G]$ is trivial.
\end{conjecture}

The assumption that $G$ is torsion-free is necessary: if $g \in G$ is an element of finite order~$n$, then $\frac{1}{n} \sum_{i=0}^{n-1} g^i$ is a nontrivial idempotent in~$\IQ[G]$.

A counterexample to the conjecture above would be in particular a zero-divisor in~$\Bbbk [G]$, and hence a counterexample to Problem~6 in Irving Kaplansky’s famous problem list~\cite{Kaplansky1956}, which is reproduced in~\cite{Kaplansky1970}.

It is interesting to notice that the analog of Conjecture~\ref{conj:Kaplansky} for the integral group ring is true for all groups, even for groups with torsion.
The proof that we give below uses operator algebras, as suggested in~\cite[page~451]{Kaplansky1970}, and therefore it is very different from the rest of this paper, even though the idea of using traces plays a central role in Section~\ref{sec:trace}.

\begin{theorem}
For any group~$G$, every idempotent in the integral group ring~$\IZ[G]$ is~trivial.
\end{theorem}

\begin{proof}
The integral group ring embeds into the reduced complex group $C^*$-algebra $C_r^* G$, and the map
$\IZ [G] \TO \IZ$, $\sum a_g g \longmapsto a_e$ extends to a positive faithful trace $\mathit{tr} \colon C_r^* G \TO \IC$. Let $p \in \IZ[G]$ be an idempotent, i.e., $p=p^2$. It is known that in the $C^*$-algebra $C_r^*G$ every idempotent is similar to a projection, i.e., there exist $q,u \in C^*_rG$ such that $q=q^2=q^*$, $u$~is invertible, and~$p = u^{-1}q u$; see for example \cite[Proposition~1.8, Lemma~1.18]{Cuntz-Rosenberg}. Therefore $\mathit{tr}(p) = \mathit{tr}(q)$. Applying the trace to $1= q + (1-q) = q^*q + (1-q)^*(1-q)$ and using positivity one sees that the trace of~$q$ lies in $[0,1]$. The trace of~$p$ is clearly an integer. Therefore $\mathit{tr} (q) =0$ or $\mathit{tr}( q ) =1$. By faithfulness of the trace this implies that $q=0$ or $q=1$, and then the same holds for $p=u^{-1}qu$.
\end{proof}

The module $Rp$ for an idempotent $p=p^2$ in the ring $R$ is an example of a finitely generated projective left $R$-module.
In view of the conjecture and the result above it seems natural to ask whether all finitely generated projective modules over group rings of torsion-free groups are necessarily free.
Again, the assumption that $G$ is torsion-free is necessary: if $g \in G$ is an element of finite order $n$, then for the non-trivial idempotent $p=\frac{1}{n} \sum_{i=0}^{n-1} g^i\in\IQ[G]$ the module $\IQ[G]p$ is projective but not free.

\begin{examples}
\label{ex:ex}
\begin{enumerate}[label=(\roman*), widest=iii]
\item
Over fields and over principal ideal domains, hence in particular over the polynomial and Laurent polynomial rings $\Bbbk [t]$ and~$\Bbbk[t^{\pm 1}]$ with coefficients in a field~$\Bbbk$, all projective modules are free.
\item \label{i:serre-conj}
The question whether finitely generated projective modules over the polynomial ring $\Bbbk[t_1 , \ldots , t_n]$ for $n \geq 2$ are necessarily free was raised by Jean-Pierre Serre\index{conjecture!Serre} in~\cite{Serre-Faisceaux}, and was answered affirmatively only 21 years later independently by Dan Quillen and Andrei Suslin. The wonderful book \cite{Lam-Serre-problem} gives a detailed account of this exciting story.
\end{enumerate}
\noindent
The polynomial ring $R[t_1, \ldots , t_n]$ is the monoid algebra $R[A]$ of the free abelian monoid $A$ generated by $\{ t_1 , \ldots , t_n \}$. The statement in \ref{i:serre-conj} was generalized as follows to monoid algebras.
\begin{enumerate}[label=(\roman*), resume]
\item \label{i:generalize-Serre}
If $R$ is a principal ideal domain, then every finitely generated projective module over the monoid algebra $R[A]$ is free
provided that $A$ is a semi-normal, abelian, cancellative monoid without nontrivial units
\cite{Gubeladze}, \cite{Swan-generalize-Serre}.
Free abelian groups are examples of monoids satisfying these conditions.
\item\label{Bass-free-groups}
If $R$ is a principal ideal domain and $F$ a finitely generated free group, then every finitely generated projective module over the group ring $R[F]$ is free \cite{Bass-free-groups}.
\end{enumerate}
At this point one could over-optimistically conjecture that every finitely generated projective $\IQ[G]$-module is free if $G$ is a torsion-free group. However:
\begin{enumerate}[label=(\roman*), resume]
\item\label{i:dunwoody}
Martin Dunwoody constructed in~\cite{Dunwoody} a torsion-free group $G$ and a finitely generated projective $\IZ[G]$-module $P$ which is not free but has the property that $P \oplus \IZ[G] \cong \IZ[G] \oplus \IZ[G]$.
There are also finitely generated projective modules over $\IQ[G]$ with analogous properties.
\end{enumerate}
\end{examples}

A weakening of the question above is whether all finitely generated projective $R[G]$-modules are induced from finitely generated projective $R$-modules when $G$ is torsion-free.
Recall that $K_0( R )$ is defined as the group completion of the monoid of isomorphism classes of finitely generated projective $R$-modules.
The surjectivity of the natural map
\[
K_0( R) \TO K_0 ( R[G] )
\]
induced by $[M] \longmapsto \bigl[ R[G] \tensor_R M \bigr]$ studies the stable version of this question:
is every finitely generated projective $R[G]$-module $P$ stably induced?
I.e., is there an $n \geq 0$ such that $P \oplus R[G]^n$ is induced from a finitely generated projective $R$-module?
Notice that this is true for Dunwoody's example~\ref{i:dunwoody} above.
The stable version of Serre's Conjecture~\ref{i:serre-conj} above is a lot easier to prove and was established much earlier in~\cite[Proposition~10]{Serre1957}.

This discussion leads to the following conjecture.
In order to formulate it, we need to recall some notions from the theory of rings.
A ring~$R$ is called left Noetherian if submodules of finitely generated left modules are always finitely generated, and it is said to have finite left global dimension if every left module has a projective resolution of finite length.
If both properties hold, then $R$ is called left regular.
In the sequel we only consider left modules and therefore simply say \emph{regular} instead of left regular.
The ring of integers~$\IZ$, all PIDs, and all fields are examples of regular rings.

\begin{conjecture}
\label{conj:FJ-K_0-regular}
Let $R$ be a regular ring, and assume that the orders of all finite subgroups of~$G$ are invertible in~$R$.
Then the map
\[
\colim_{H\in\obj\Sub G(\Fin)}
K_0(R[H])
\TO[\cong]
K_0(R[G])
\]
is an isomorphism.
In particular, if $G$ is torsion-free, then for any regular ring~$R$ there is an isomorphism
\[
K_0(R)\TO[\cong]K_0(R[G])
\,.
\]
\end{conjecture}
Here the colimit is taken over the finite subgroup category~$\Sub G(\Fin)$, whose objects are the finite subgroups $H$ of~$G$ and whose morphisms are
defined as follows.
Given finite subgroups $H$ and~$K$ of~$G$, let $\conhom_G(H,K)$ be the set of all group homomorphisms $H\TO K$ given by conjugation by an element of~$G$.
The group $\inn(K)$ of inner automorphisms of~$K$ acts on~$\conhom_G(H,K)$ by post-composition.
The set of morphisms in~$\Sub G(\Fin)$ from $H$ to~$K$ is then defined as the quotient~$\conhom_G(H,K)/\inn(K)$. Since inner conjugation induces the identity on $K_0( R[ - ] )$, this is indeed a well defined functor on $\Sub G ( \Fin )$.
In the special case when $G$ is abelian, the category $\Sub G(\Fin)$ is just the poset of finite subgroups of~$G$ ordered by inclusion.

\begin{proposition} \label{4implies1}
Conjecture~\ref{conj:FJ-K_0-regular} implies Conjecture~\ref{conj:Kaplansky}.
\end{proposition}

\begin{proof}
Let $\Bbbk$ be a field of characteristic zero and let $G$ be a torsion-free group.
Let~$\MOR{\epsilon}{\Bbbk[G]}{\Bbbk}$ denote the augmentation and write $\epsilon_*M=\Bbbk\tensor_{\Bbbk[G]}M$.
If $p\in\Bbbk[G]$ is an idempotent, then
\[
\Bbbk[G]\cong\Bbbk[G]p\oplus\Bbbk[G](1-p)
\AND
\Bbbk\cong\epsilon_*\Bbbk[G]\cong\epsilon_*\Bbbk[G]p\oplus\epsilon_*\Bbbk[G](1-p)
\,.
\]
Since $\Bbbk$ is a field, either $\epsilon_*\Bbbk[G]p$ or $\epsilon_*\Bbbk[G](1-p)$ is the zero module.
Replacing $p$ by~$1-p$ if necessary, let us assume that $\epsilon_*\Bbbk[G]p$ is zero.
The assumption $\IZ\cong K_0(\Bbbk)\cong K_0(\Bbbk[G])$ implies that there exist $n$ and $m$ such that
\[
\Bbbk[G]p\oplus\Bbbk[G]^n\cong\Bbbk[G]^m.
\]
Applying $\epsilon_*$ we see that $n=m$, and from this we conclude that $\Bbbk[G]p$ is zero as follows.
Recall that a ring $R$ is called \emph{stably finite} if $M\oplus R^n\cong R^n$ always implies that $M$ is zero; see~\cite[Section~1B]{Lam}.
Kaplansky showed that, if $\Bbbk$ is a field of characteristic~$0$, then any group ring~$\Bbbk[G]$ is stably finite; compare~\cite{Montgomery}.
\end{proof}



\pagebreak

\subsection{h-Cobordisms}
\label{subsec:h-cob}

Recall that a smooth cobordism over a closed $n$-dimensional smooth manifold~$M$ consists of another closed $n$-dimensional smooth manifold~$N$ and an $(n+1)$-dimensional compact smooth manifold~$W$ with boundary~$\partial W$ together with a diffeomorphism $\MOR[\cong]{(f,g)}{M\amalg N}{\partial W}$.
This is called an $h$-cobordism if both $\incl\circ f$ and~$\incl\circ g$ are homotopy equivalences, where $\incl$ denotes the inclusion of~$\partial W$ in~$W$.
Two cobordisms $W$ and~$W'$ over~$M$ are called isomorphic if there exists a diffeomorphism $\MOR[\cong]{F}{W}{W'}$ such that $F_{|\partial W}\circ f=f'$.
A cobordism over~$M$ is called trivial if it is isomorphic to the cylinder~$M\times[0,1]$ (and this in particular implies that $M$ and~$N$ are diffeomorphic).

\begin{conjecture}[trivial $h$-cobordisms]
\label{conj:trivial-h-cobordisms}
\index{conjecture!trivial h-cobordisms@trivial $h$-cobordisms}
Let $M$ be a closed, connected, smooth manifold of dimension at least~$5$ and with torsion-free fundamental group.
Then every $h$-cobordism over~$M$ is trivial.
\end{conjecture}

Surprisingly, this conjecture can be reinterpreted in terms of algebraic $K$-theory.
In fact, the celebrated $s$-Cobordism Theorem of Stephen Smale, Barry Mazur, John Stallings, and Dennis Barden (e.g., see~\cite{Milnor-h-cobordism}, \cite{Kreck-Lueck-Novikov}), states that there is a bijection
\[
\{\,\text{$h$-cobordisms over~$M$}\,\}/\text{iso}\, \cong \wh(\pi_1(M))
\]
between the set of isomorphism classes of smooth cobordisms over~$M$ and the Whitehead group~$\wh(\pi_1(M))$ of the fundamental group of~$M$, whose definition we now review.

Recall that, given a ring~$R$, invertible matrices with coefficients in~$R$ represent classes in~$K_1(R)$.
Given any group~$G$, the elements $\pm g\in\IZ[G]$ are invertible for any~$g\in G$, and hence represent elements in $K_1(\IZ[G])$.
By definition the Whitehead group $\wh(G)$ is the quotient of $K_1( \IZ [G] )$ by the image of the map that sends $(\pm1,g)$ to the element represented by $\pm g$ in $K_1(\IZ[G])$.
This map factors over $\{\pm1\}\oplus G_{ab}$, where $G_{ab}$ is the abelianization of~$G$, and
the induced map $\{\pm1\}\oplus G_{ab}\TO K_1(\IZ[G])$ is in fact injective; see for example \cite[Lemma~2]{LR-survey}. 
So there is a short exact sequence
\begin{equation}
\label{eq:ses-Wh}
0 \TO \{ \pm 1 \} \oplus G_{ab} \TO K_1( \IZ [G] ) \TO \wh(G) \TO 0
\,.
\end{equation}

For Whitehead groups there is the following well-known folklore conjecture.
The cases of the infinite cyclic group~\cite{Higman1940}, of finitely generated free abelian groups~\cite{BHS}, and of finitely generated free groups~\cite{Stallings} provided early evidence for this conjecture.

\begin{conjecture}
\label{conj:Wh(tor-free)=0}
If $G$ is a torsion-free group, then $\wh(G)=0$.
\end{conjecture}

By the $s$-Cobordism Theorem recalled above, the connection between the last two conjectures is as follows.

\begin{proposition} \label{6equivalent8}
Let $M$ be a closed, connected, smooth manifold of dimension at least~$5$ and with torsion-free fundamental group.
Then Conjecture~\ref{conj:trivial-h-cobordisms} for~$M$ is equivalent to Conjecture~\ref{conj:Wh(tor-free)=0} for~$G=\pi_1(M)$.
\end{proposition}

For groups with torsion, the situation is much more complicated.
For example, if $C_n$ is a finite cyclic group of order $n\not\in\{1,2,3,4,6\}$, then $\wh(C_n)\ne0$, and in fact even $\wh(C_n)\tensor_\IZ\IQ\ne0$.
The analog of Conjecture~\ref{conj:Wh(tor-free)=0} for arbitrary groups is the following.

\begin{conjecture}
\label{conj:Wh-inj}
For any group~$G$ the map
\begin{equation}
\label{eq:Wh}
\colim_{H\in\obj\Sub G(\Fin)}
\wh(H)\tensor_\IZ\IQ
\TO
\wh(G)\tensor_\IZ\IQ
\end{equation}
is injective.
\end{conjecture}

We highlight two differences with the corresponding Conjecture~\ref{conj:FJ-K_0-regular} for~$K_0$.
First, Conjecture~\ref{conj:Wh-inj} is only a rational statement, i.e., after applying~$-\tensor_\IZ\IQ$.
Second, it is only an injectivity statement.
In order to obtain a rational isomorphism conjecture for~$\wh(G)$ one needs to enlarge the source of the map~\eqref{eq:Wh}.
This requires some additional explanations and is postponed to Conjecture~\ref{conj:Wh-iso} below.


\subsection{Assembly maps}
\label{subsec:assembly}

The Farrell-Jones Conjecture, which we formulate in the next subsection, generalizes Conjectures~\ref{conj:FJ-K_0-regular}, \ref{conj:Wh(tor-free)=0}, and~\ref{conj:Wh-inj} from statements about the abelian groups $K_0$ and~$\wh$ to statements about the non-connective algebraic $K$-theory spectra~$\K(R[G])$ of group rings, for arbitrary coefficient rings and arbitrary groups.
In order to formulate the Farrell-Jones Conjecture, we need to first introduce the fundamental concept of assembly maps.

Fixing a ring~$R$, algebraic $K$-theory defines a functor~$\K(R[-])$ from groups to spectra.
In fact, it is very easy to promote this to a functor
\[
\MOR{\K(R[-])}{\Groupoids}{\Sp}
\]
from the category of small groupoids (i.e., small categories whose morphisms are all isomorphisms) to the category of spectra.
Moreover, this functor preserves equivalences, in the sense that it sends equivalences of groupoids to $\pi_*$-isomorphisms (i.e., weak equivalences) of spectra.
For any such functor we now proceed to construct assembly maps, following the approach of~\cite{Davis-L}.
It is not enough to work in the stable homotopy category of spectra, but any point-set level model would work.

Let
\(
\MOR{\T}{\Groupoids}{\Sp}
\)
be a functor that preserves equivalences.
Given a group~$G$, consider the functor~$\MOR{\oid{G}{-}}{\Sets^G}{\Groupoids}$ that sends a $G$-set~$S$ to its action groupoid~$\oid{G}{S}$, with $\obj\oid{G}{S}=S$ and \(\mor_{\oid{G}{S}}(s,s')=\SET{g\in G}{gs=s'}\).
Restricting to the orbit category~$\Or G$, i.e., the full subcategory of~$\Sets^G$ with objects~$G/H$ as $H$ varies among the subgroups of~$G$, we obtain the horizontal composition in the following diagram.
\[
\begin{tikzcd}[cramped]
\Or G
\arrow[d, hook, "\iota"']
\arrow[r, hook]
&
\Sets^G
\arrow[r, "\oid{G}{-}"]
&
\Groupoids
\arrow[r, "\T"]
&
\Sp \\
\Top^G
\arrow[urrr, bend right=10, dotted, "\Lan_\iota\T(\oid{G}{-})" description]
& & &
\end{tikzcd}
\]
Now we take the left Kan extension \cite[Section~X.3]{MacLane}
of $\T(\oid{G}{-})$ along the full and faithful inclusion functor~$\iota\colon\Or{G}\hookrightarrow\Top^G$ of $\Or{G}$ into the category of all $G$-spaces.
The left Kan extension evaluated at a $G$-space $X$ can be constructed  as the coend \cite[Sections~IX.6 and~X.4]{MacLane}
\[
\bigl(\Lan_\iota\T(\oid{G}{-})\bigr)
(X)
=
\smash{X_+\sma_{\Or G}\T(\oid{G}{-})}
\]
of the functor
\begin{align*}
(\Or G)^\op\times\Or G
&\TO
\Sp
\,,
\\
(G/H,G/K)
&\longmapsto
\map(G/H,X)^G_+\sma\T(\oid{G}{\,(G/K)})\cong X^H_+\sma\T(\oid{G}{\,(G/K)})
\,.
\end{align*}
There are natural isomorphisms $G/H_+\sma_{\Or G}\T(\oid{G}{-})\cong\T(\oid{G}{\,(G/H)})$, and the fact that $\T$ preserves equivalences implies that these spectra are $\pi_*$-isomorphic to~$\T(H)$.
Notice that for $\pt=G/G$ we even have an isomorphism $\pt_+\sma_{\Or G}\T(\oid{G}{-})\cong\T(G)$.

To define the assembly map we apply this construction to the following $G$-spaces.
Consider a family~$\CF$ of subgroups of~$G$ (i.e., a collection of subgroups closed under passage to subgroups and conjugates)
and consider a universal $G$-space $EG(\CF)$.
This is a $G$-CW~complex characterized up to $G$-homotopy equivalence by the property that, for any subgroup~$H\le G$, the $H$-fixed point space
\[
\bigl(EG(\CF)\bigr)^H
\,
\text{ is }
\begin{cases}
\text{empty}&\text{if $H\not\in\CF$;}\\
\text{contractible}&\text{if $H\in\CF$.}
\end{cases}
\]
The \emph{assembly map}\index{assembly map} is by definition the map
\[
\MOR{\asbl_\CF}{EG(\CF)_+\sma_{\Or G}\T(\oid{G}{-})}{\T(G)}
\]
induced by the projection $EG(\CF)\TO\pt$ (where, in the target, we use the isomorphism $\pt_+\sma_{\Or G}\T(\oid{G}{-})\cong\T(G)$).

\begin{remark}
\leavevmode
\label{rem:source-assembly}
\begin{enumerate}[label=(\roman*)]
\item \label{rem-i:classical}
In the special case of the trivial family~$\CF=1$, a universal space~$EG(1)$ is by definition a free and non-equivariantly contractible $G$-CW~complex, i.e., the universal cover of a classifying space~$BG$.
In this case, there is an identification
\[EG(1)_+\sma_{\Or G}\T(\oid{G}{-})\cong BG_+\sma\T(1)\]
and therefore we obtain the so-called \emph{classical} assembly map\index{assembly map!classical}
\[
\MOR{\asbl_1}{BG_+\sma\T(1)}{\T(G)}
\,.
\]
\item \label{rem-i:rel-assembly}
Any $G$-CW~complex whose isotropy groups all lie in the family $\CF$ has a map to~$EG(\CF)$, and this map is unique up to $G$-homotopy.
This applies in particular to $EG(\CF')$ when $\CF'\subseteq\CF$, and we refer to the induced map
\[
\MOR{\asbl_{\CF' \subseteq \CF}}{EG(\CF')_+\sma_{\Or G}\K(R[\oid{G}{-}])}{EG(\CF)_+\sma_{\Or G}\K(R[\oid{G}{-}])}
\]
as the \emph{relative} assembly map\index{assembly map!relative}.
\item \label{rem-i:hocolim}
The source of the assembly map is a model for
\[
\hocolim_{\substack{G/H\in\obj\Or G\\\text{s.t. }H\in\CF}}\T(\oid{G}{\,(G/H)})
\,,
\]
the homotopy colimit of the restriction of~$\T(\oid{G}{-})$ to the full subcategory of~$\Or G$ of objects $G/H$ with $H\in\CF$; compare~\cite[Section~5.2]{Davis-L}.
\item \label{rem-i:AHSS}
Taking the homotopy groups of~$X_+\sma_{\Or G}\T(\oid{G}{-})$ defines a $G$-equivariant homology theory for $G$-CW complexes~$X$.
This is an equivariant generalization of the well-known statement that $\pi_*(X_+\sma\BE)$ gives a non-equivariant homology theory for any spectrum~$\BE$.
The Atiyah-Hirzebruch spectral sequence converging to~$\pi_{s+t}(X_+\sma\BE)$ with $E^2_{s,t}=H_s(X;\pi_t\BE)$ also generalizes to a spectral sequence converging to~$\pi_{s+t}(X_+\sma_{\Or G}\T(\oid{G}{-}))$ with
\[
E^2_{s,t}=H^G_s(X;\pi_t\BT(\oid{G}{-}))
\,,
\]
the Bredon homology of~$X$ with coefficients in $\MOR{\pi_t\BT(\oid{G}{-})}{\Or G}{\Ab}$; compare \cite[Theorem~4.7]{Davis-L}. 
Using this we see that, if $\asbl_{\CF}$ is a $\pi_*$-isomorphism, then in general all $\pi_t( \BT ( H ))$ with $H \in \CF$ and $- \infty < t \leq n$ contribute to $\pi_n( \BT ( G))$.
\end{enumerate}
\end{remark}

We conclude with a historical comment.
The classical assembly map $\asbl_1$ from Remark~\ref{rem:source-assembly}\ref{rem-i:classical} for algebraic $K$-theory was originally introduced in Jean-Louis Loday's thesis~\cite[Chapitre IV]{Loday} using pairings in algebraic $K$-theory and the multiplication map
\[
G \times GL(R) \TO GL( R[G] )
\,.
\]
Friedhelm Waldhausen \cite[Section~15]{WaldhausenII} characterized this map as a universal approximation by a homology theory evaluated on a classifying space. This point of view was nicely explained by Michael Weiss and Bruce Williams in~\cite{Weiss-Williams}.
In their original work~\cite{FJ-iso}, Farrell and Jones used the language developed by Frank Quinn~\cite[Appendix]{Quinn-EndsII}. Later, Jim Davis and Wolfgang L\"uck~\cite{Davis-L} gave an equivariant version of the point of view of~\cite{Weiss-Williams}, clarifying and unifying the underlying principles. Their approach leads to the concise description of the assembly map given above.
The different approaches are compared and shown to agree in~\cite{Hambleton-Pedersen}.


\subsection{The Farrell-Jones Conjecture}
\label{subsec:FJ}

We begin by formulating the Farrell-Jones Conjecture in the special case of torsion-free groups and regular rings.

\begin{conjecture-FJ}[special case]
\label{conj:FJ-tor-free-regular}
\index{conjecture!Farrell-Jones!special case of torsion-free groups and regular rings}
For any \emph{torsion-free} group~$G$ and for any \emph{regular} ring~$R$ the classical assembly map
\[
\MOR{\asbl_1}{BG_+\sma\K(R)}{\K(R[G])}
\]
is a $\pi_*$-isomorphism.
\end{conjecture-FJ}

On $\pi_0$ the classical assembly map produces the map $K_0(R)\TO K_0(R[G])$ induced by the inclusion $R\TO R[G]$.
So we see that the Farrell-Jones Conjecture~\ref{conj:FJ-tor-free-regular} implies the torsion-free case of Conjecture~\ref{conj:FJ-K_0-regular}.

On~$\pi_1$, in the special case when~$R=\IZ$, we have
\begin{equation}
\label{eq:H_1(BG;KZ)}
\pi_1\bigl(BG_+\sma\K(\IZ)\bigr)
\cong
H_0\bigl(BG;K_1(\IZ)\bigr)
\oplus
H_1\bigl(BG;K_0(\IZ)\bigr)
\cong
\{\pm1\}
\oplus
G_{ab}
\,.
\end{equation}
The first isomorphism comes from the Atiyah-Hirzebruch spectral sequence, which is concentrated in the first quadrant because regular rings have vanishing negative $K$-theory.
The second isomorphism comes from the 
computations $K_1(\IZ)\cong\{\pm1\}$ and~$K_0(\IZ)\cong\IZ$.
Under the isomorphism~\eqref{eq:H_1(BG;KZ)}, it can be shown \cite[Assertion~15.8]{WaldhausenII} that the classical assembly map produces on~$\pi_1$ the left-hand map in~\eqref{eq:ses-Wh}, whose cokernel is by definition the Whitehead group~$\wh(G)$.
So we see that the Farrell-Jones Conjecture~\ref{conj:FJ-tor-free-regular} implies Conjecture~\ref{conj:Wh(tor-free)=0}.

From these identifications and
computations of $K_0( \IZ [G] )$ and $Wh(G)$ for finite groups we see that  $\pi_0( \asbl_1)$ and $\pi_1( \asbl_1)$ may not be surjective for groups with torsion, even when $R = \IZ$.
The classical assembly map may also fail to be injective on homotopy groups if we drop the assumption torsion-free. This happens for example for $\pi_2( \asbl_1)$ if $R= \IF$ is a
finite field of characteristic prime to $2$ and $G$ is the non-cyclic group with  $4$ elements \cite{Ullmann-Wu}.

The regularity assumption cannot be dropped either.
For example, consider the case when~$G=C_\infty$ is the infinite cyclic group.
Then of course $BC_\infty=S^1$ and~$R[C_\infty]=R[t,t^{-1}]$, and it can be shown that on~$\pi_n$ the classical assembly map produces the left-hand map in the short exact sequence
\[
0\TO K_n(R)\oplus K_{n-1}(R)\TO K_n(R[t,t^{-1}])\TO NK_n(R)\oplus NK_n(R)\TO0
\]
given by the Fundamental Theorem of algebraic $K$-theory;
see for example~\cite{BHS} in low dimensions, \cite[Section~10]{Swan-higher}, and~\cite[Theorem~18.1]{WaldhausenII}.
Recall that the groups $NK_n(R)$ are defined as the cokernel of the split injection $K_n(R) \TO K_n(R[t])$ induced by the natural map $R \TO R[t]$.
It is known that $NK_n(R)=0$ for each~$n$ if $R$ is regular~\cite[Theorem~10.1(1) and 10.3]{Swan-higher}, but $NK_n (R)$ can be nontrivial for arbitrary rings.
So we see that the classical assembly map for the infinite cyclic group is a $\pi_*$-isomorphism if the ring~$R$ is regular, but otherwise it may fail to be surjective on homotopy groups.

For arbitrary groups and rings, the generalization of Conjecture~\ref{conj:FJ-tor-free-regular} is the following.

\begin{conjecture-FJ}
\label{conj:FJ}
\index{conjecture!Farrell-Jones}
For any group~$G$ and for any ring~$R$ the Farrell-Jones assembly map
\[
\MOR{\asbl_\VCyc}{EG(\VCyc)_+\sma_{\Or G}\K(R[\oid{G}{-}])}{\K(R[G])}
\]
is a $\pi_*$-isomorphism.
\end{conjecture-FJ}
Here $\VCyc$ denotes the family of virtually cyclic subgroups of $G$. A group is called virtually cyclic if it contains a cyclic subgroup of finite index.

The Farrell-Jones Conjectures~\ref{conj:FJ-tor-free-regular} and~\ref{conj:FJ} are related as follows.

\begin{proposition} \label{prop:FJregtfFJ}
If $G$ is a torsion-free group and $R$ is a regular ring, then the Farrell-Jones Conjectures \ref{conj:FJ-tor-free-regular} and~\ref{conj:FJ} are equivalent.
\end{proposition}

\begin{proof}
This is an application of the following principle, which is proved in~\cite[Theorem~65]{LR-survey}. 

\begin{transitivity}
\label{transitivity}
\index{assembly map!transitivity principle}
Let $\CF$ and $\CF'$ be families of subgroups of $G$ with $\CF \subseteq \CF'$. Assume that for each $H \in \CF'$ the assembly map
 \[
EH(\CF|_{H})_+\sma_{\Or H}\K(R[\oid{H}{-}])
\TO
\K(R[H])
\]
is a $\pi_*$-isomorphism, where $\CF|_H=\SET{K\leq H}{K\in\CF}$.
Then the relative assembly map explained in Remark~\ref{rem:source-assembly}\ref{rem-i:rel-assembly}, i.e., the left vertical map in the following commutative triangle, is a $\pi_*$-isomorphism.
\[
\begin{tikzcd}[row sep=small]
\ds EG(\CF)_+\sma_{\Or G}\K(R[\oid{G}{-}])
\arrow[drr,"\asbl_{\CF}" pos=.25, shorten <=-.7em]
\arrow[dd,"\asbl_{\CF\subseteq\CF'}"']
& &
\\
& &
\K (R[G])
\\
\ds EG(\CF')_+\sma_{\Or G}\K(R[\oid{G}{-}])
\arrow[urr,"\asbl_{\CF'}"' pos=.4, shorten <=.5em]
& &
\end{tikzcd}
\]
Therefore,
$\asbl_{\CF}$ is a $\pi_*$-isomorphism if and only if $\asbl_{\CF'}$ is a $\pi_*$-isomorphism.
\end{transitivity}

We now apply the transitivity principle in the case $\CF=1$ and $\CF' = \VCyc$.
Any nontrivial torsion-free virtually cyclic group is infinite cyclic.
Recall that $\asbl_1$ can be identified with the classical assembly map in Conjecture~\ref{conj:FJ-tor-free-regular}.
So it is enough to show that the classical assembly map is a $\pi_*$-isomorphism for the infinite cyclic group $C$.
The fact that this is true in the case of regular rings is explained above, before the statement of Conjecture~\ref{conj:FJ}.
\end{proof}

The next result shows, as promised, that the Farrell-Jones Conjecture implies all the other conjectures introduced in the first two subsections; the case of Conjecture~\ref{conj:Wh-inj} is considered right after Conjecture~\ref{conj:Wh-iso} below.

\begin{proposition}
The Farrell-Jones Conjecture~\ref{conj:FJ} implies Conjectures \ref{conj:FJ-K_0-regular} and \ref{conj:Wh(tor-free)=0}, and so also Conjectures \ref{conj:Kaplansky} and~\ref{conj:trivial-h-cobordisms} by Propositions \ref{4implies1} and~\ref{6equivalent8}.
\end{proposition}

\begin{proof}
The case of Conjecture~\ref{conj:Wh(tor-free)=0} and the torsionfree case of Conjecture~\ref{conj:FJ-K_0-regular} is explained above, directly after the statement of Conjecture~\ref{conj:FJ-tor-free-regular}.
The general case of Conjecture~\ref{conj:FJ-K_0-regular} follows from the following isomorphisms.
\begin{align*}
\pi_0 \left( EG(\VCyc)_+\sma_{\Or G} \K(R[\oid{G}{-}]) \right)
& \stackrel{{\ts\one}}{\cong}
\pi_0 \left( EG(\Fin)_+\sma_{\Or G} \K(R[\oid{G}{-}]) \right)
\\
& \stackrel{{\ts\two}}{\cong}
\pi_0 \left( EG(\Fin)_+\sma_{\Or G ( \Fin )} \K(R[\oid{G}{-}]) \right)
\\
& \stackrel{{\ts\three}}{\cong}
H_0 \left( C( EG(\Fin)) \tensor_{\IZ \Or G (\Fin)} K_0(R[\oid{G}{-}]) \right)
\\
& \stackrel{{\ts\four}}{\cong}
{\IZ} \tensor_{\IZ \Or G (\Fin)} K_0(R[\oid{G}{-}])
\\
& \stackrel{{\ts\five}}{\cong}
\colim_{\Or G (\Fin)} K_0(R[\oid{G}{-}])
\\
& \stackrel{{\ts\six}}{\cong}
\colim_{\Sub G (\Fin)} K_0(R[-])
\end{align*}
Theorem~\ref{thm:Fin-to-VCyc}\ref{i:field-char0} yields the isomorphism $\one$. Since
$EG( \Fin )^H = \emptyset$
if $H$ is not finite,
the isomorphism $\two$ follows immediately by inspecting the construction of the coend. The assumptions that $R$ is regular and that the order of every finite subgroup $H$ of~$G$ is invertible in~$R$ imply that also $R[H]$ is regular. For regular rings the negative $K$-groups vanish \cite[3.3.1]{Rosenberg}, 
and therefore the equivariant Atiyah-Hirzebruch spectral sequence explained in Remark~\ref{rem:source-assembly}\ref{rem-i:AHSS} is concentrated in the first quadrant.
This gives the isomorphism $\three$.
The singular or cellular chain complex $C( EG ( \Fin ))$, considered as a contravariant functor $G/H \longmapsto C( EG( \Fin )^H )$, resolves the constant functor~$\IZ$, therefore $\four$ follows from right exactness of~$-\tensor_{\IZ\Or G(\Fin)}M$ for any fixed~$\MOR{M}{\Or G(\Fin)}{\Ab}$.
The coend with the constant functor~$\IZ$  is one possible construction of the colimit in abelian groups, hence $\five$.
Since $K_n( R[\oid{G}{G/H}]) \cong K_n(R[H])$ and since inner automorphisms induce the identity on $K$-theory, the functor $K_n (R [ \oid{G} - ])$ factors over $\Or G (\Fin) \TO \Sub G ( \Fin)$, the functor sending $G/H \to G/K$, $gH \mapsto gaH$ to the class of $H \to K$, $h \mapsto a^{-1} h a$.
The isomorphism~$\six$ then follows by standard properties of colimits.
\end{proof}

The next result deals with the passage from finite to virtually cyclic subgroups in the source of the Farrell-Jones assembly map.

\begin{theorem}[finite to virtually cyclic]
\leavevmode
\label{thm:Fin-to-VCyc}
\index{conjecture!Farrell-Jones!from finite to virtually cyclic subgroups}
\begin{enumerate}[label=(\roman*)]
\item
\label{i:Bartels-domain}
The relative assembly map
\[
\MOR{\asbl_{\Fin\subseteq\VCyc}}%
{EG(\Fin )_+\sma_{\Or G}\K(R[\oid{G}{-}])}%
{EG(\VCyc)_+\sma_{\Or G}\K(R[\oid{G}{-}])}
\]
is always split injective.
\item
\label{i:field-char0}
If $R$ is regular and the order of every finite subgroup of $G$ is invertible in $R$, then $\asbl_{\Fin\subseteq\VCyc}$ is a $\pi_*$-isomorphism.
\item
\label{i:L-Steimle}
If $R$ is regular then $\asbl_{\Fin\subseteq\VCyc}$ is a $\pi_*^\IQ$-isomorphism, i.e., it induces isomorphisms on~$\pi_n(-)\tensor_\IZ\IQ$ for all~$n\in\IZ$.
\end{enumerate}
\end{theorem}

\begin{proof}
Part~\ref{i:Bartels-domain} is the main result of~\cite{Bartels-domain}.
Part~\ref{i:field-char0} is shown in~\cite[Proposition~70]{LR-survey}. 
Part~\ref{i:L-Steimle} is proved in~\cite[Theorem~0.2]{L-Steimle} 
and generalizes~\cite[Corollary on page 165]{Grunewald}.
\end{proof}


\subsection{Rational computations}

After tensoring with the rational numbers, the Farrell-Jones Conjecture~\ref{conj:FJ} for regular rings can be reformulated in a more concrete and computational fashion as follows.

Assume that $R$ is a regular ring.
Recall from Theorem~\ref{thm:Fin-to-VCyc}\ref{i:L-Steimle} that the relative assembly map~$\asbl_{\Fin\subseteq\VCyc}$ induces isomorphisms
\begin{equation}
\label{eq:rel-asbl-rationally}
\pi_n\Bigl({EG(\Fin )_+\sma_{\Or G}\K(R[\oid{G}{-}])}\Bigr)\tensor_\IZ\IQ
\TO[\cong]
\pi_n\Bigr({EG(\VCyc)_+\sma_{\Or G}\K(R[\oid{G}{-}])}\Bigr)\tensor_\IZ\IQ
\,.
\end{equation}
The theory of equivariant Chern characters developed by L\"uck in~\cite{L-Chern} yields the following isomorphisms:
\begin{equation}
\label{eq:Lueck-Grunewald}
\begin{tikzcd}[row sep=small]
\ds
\adjustlimits\bigoplus_{(C)\in(\FCyc)}
\bigoplus_{s+t=n}
H_s(BZ_GC;\IQ)\tensor_{\IQ[W_G C]}\Theta_C\Bigl(K_t(R[C])\tensor_{\IZ}\IQ\Bigr)
\arrow[d, "\cong", pos=.1, shorten <=-.75em]
\\
\ds
\pi_n\Bigl(EG(\FCyc)_+\sma_{\Or G}\K(R[\oid{G}{-}])\Bigr)\tensor_\IZ\IQ
\arrow[d, "\cong", pos=.7, shorten >=-.5em]
\\
\ds
\pi_n\Bigl(EG(\Fin)_+\sma_{\Or G}\K(R[\oid{G}{-}])\Bigr)\tensor_\IZ\IQ
\,.
\end{tikzcd}
\end{equation}
Before we explain the notation, notice the analogy with the well-known isomorphism
\[
\bigoplus_{s+t=n}H_s(BG;\IQ)\tensor_\IZ\Bigl(K_t(R)\tensor_{\IZ}\IQ\Bigr)
\TO[\cong]
\pi_n\bigl(BG_+\sma\K(R)\bigr)\tensor_\IZ\IQ
\,,
\]
whose source corresponds to the summand in~\eqref{eq:Lueck-Grunewald} indexed by~$C=1$.

Given a subgroup $H$ of~$G$, we denote by $N_G H$ the normalizer and by $Z_G H$ the centralizer of $H$ in $G$, and we define the Weyl group as the quotient $W_G H = N_G H / (Z_G H \cdot H)$.
Notice that the Weyl group $W_G H$ of a finite subgroup $H$ is always finite, since it embeds into the outer automorphism group of~$H$.
We write $\FCyc$ for the family of finite cyclic subgroups of $G$, and $( \FCyc )$ for the set of conjugacy classes of finite cyclic subgroups.
Furthermore, $\Theta_C$ is an idempotent endomorphism of $K_t(R[C])\tensor_{\IZ}\IQ$, which corresponds to a specific idempotent in the rationalized Burnside ring of~$C$, and whose image is a direct summand of $K_t(R[C])\tensor_{\IZ}\IQ$ isomorphic to
\begin{equation}
\label{eq:Artin-defect}
\coker\Biggl(\, \bigoplus_{D\lneqq C} \ind_D^C\colon \bigoplus_{D\lneqq C}
     K_t(R[D])\tensor_{\IZ}\IQ \TO K_t(R[C])\tensor_{\IZ}\IQ \Biggr).
\end{equation}
The Weyl group acts via conjugation on $C$ and hence on $\Theta_C ( K_t(R[C])\tensor_{\IZ}\IQ)$.
The Weyl group action on the homology groups in the source of~\eqref{eq:Lueck-Grunewald} comes from the fact that $EN_G C / Z_G C$ is a model for $BZ_G C$.

\pagebreak

\begin{conjecture-FJ}[rationalized version]
\label{conj:FJ-rationalized}
\index{conjecture!Farrell-Jones!rationalized version}
For any group~$G$ and for any \emph{regular} ring~$R$ the composition of the Farrell-Jones assembly map and the isomorphisms~\eqref{eq:Lueck-Grunewald} and~\eqref{eq:rel-asbl-rationally}
\[
\adjustlimits\bigoplus_{(C)\in(\FCyc)}
\bigoplus_{s+t=n}
H_s(BZ_GC;\IQ)\tensor_{\IQ[W_G C]}\Theta_C\Bigl(K_t(R[C])\tensor_{\IZ}\IQ\Bigr)
\TO
K_n(R[G])\tensor_\IZ\IQ
\]
is an isomorphism for each~$n\in\IZ$.
\end{conjecture-FJ}

Analogously one obtains the following conjecture for Whitehead groups, which is the correct generalization of Conjecture~\ref{conj:Wh-inj} mentioned at the end of Subsection~\ref{subsec:h-cob}.

\begin{conjecture}
\label{conj:Wh-iso}
For any group~$G$ there is an isomorphism
\[
\begin{tikzcd}[row sep=small]
\scalebox{.95}{\(\negthickspace\negthickspace\ds
\bigoplus_{(C)\in(\FCyc)}
\Biggl(
          \IQ \tensor_{\IQ[W_G C]}\Theta_C\Bigl(        \wh(C)\tensor_{\IZ}\IQ\Bigr)
\ \oplus\
H_2(BZ_GC;\IQ)\tensor_{\IQ[W_G C]}\Theta_C\Bigl(K_{-1}(\IZ[C])\tensor_{\IZ}\IQ\Bigr)
\Biggr)
\negthickspace\negthickspace\negthickspace\)}
\arrow[d, "\cong", pos=.1, shorten <=-.85em]
\\
\ds
\wh(G)\tensor_\IZ\IQ
\,.
\end{tikzcd}
\]
\end{conjecture}

Conjecture~\ref{conj:Wh-iso} implies Conjecture~\ref{conj:Wh-inj}, because in fact
\[
\colim_{H\in\obj\Sub G(\Fin)}
\wh(H)\tensor_\IZ\IQ
\quad\cong\quad
\bigoplus_{(C)\in(\FCyc)}
          \IQ \tensor_{\IQ[W_G C]}\Theta_C\Bigl(        \wh(C)\tensor_{\IZ}\IQ\Bigr)
\]
and the map~\eqref{eq:Wh} coincides with the restriction to this summand of the map in Conjecture~\ref{conj:Wh-iso}.

\begin{remark}
For finite groups $H$ we have that $\wh(H) \tensor_{\IZ} \IQ \cong K_1( \IZ [H]) \tensor_{\IZ} \IQ$ by the exact sequence \eqref{eq:ses-Wh}.
The only difference between the sources of the maps in Conjectures~\ref{conj:FJ-rationalized} and~\ref{conj:Wh-iso} is the absence from~\ref{conj:Wh-iso} of the summands with $(s,t)=(1,0)$.
For finite groups~$H$ the natural map $\IQ \cong K_0( \IZ) \tensor_{\IZ} \IQ \TO K_0( \IZ[ H]) \tensor_{\IZ} \IQ$ is an isomorphism, and hence it follows from~\eqref{eq:Artin-defect} that the only non-vanishing summand among these is $H_1(BG ; \IQ) \cong G_{ab} \tensor_{\IZ} \IQ$ corresponding to~$C=1$. This is consistent with the exact sequence $\eqref{eq:ses-Wh}$.
\end{remark}

Finally, we note that in the special case when $R=\IZ$ the dimensions of the $\IQ$-vector spaces in~\eqref{eq:Artin-defect} for any~$t$ and any finite cyclic group~$C$ can be explicitly computed as follows.

\begin{theorem}
\label{Patronas}
Let $C$ be a cyclic group of order~$c$.
Then
\[
\dim_{\IQ}
\Theta_C\Bigl( K_t(\IZ[C])\tensor_\IZ\IQ \Bigr)
=
\begin{cases}
s(c) -1 & \text{if $t=-1$;}\\
\totient(c)/2 -1 & \text{if $t=1$ and $c>2$;}\\
1 & \text{if $t>1$, $t \equiv 1\!\!\!\mod{4}$, and $c = 2$;}\\
\totient(c)/2  & \text{if $t>1$, $t \equiv 1\!\!\!\mod{2}$, and $c > 2$;}\\
0 & \text{otherwise.}
\end{cases}
\]
Here $\totient(c)=\#\SET{x\in C}{\text{$x$ generates~$C$}}$ is Euler's $\totient$-function, $c = \prod_{i=1}^s p_i^{e_i}$ is the prime factorization of~$c$, and $s(c) = \sum_{i=1}^s \totient ( n / p_i^{e_i} ) / f_{p_i}$, where $f_{p_i}$ is the smallest number such that
$p_i^{f_{p_i}} \equiv 1\!\!\!\mod{n/ p^{e_i}}$.
\end{theorem}

This result is proved in~\cite[Theorem on page~9]{Patronas}, and more details will appear in~\cite{PRV}.


\subsection{Some related conjectures}
\label{subsec:related-conjectures}

We now survey very briefly some other conjectures that are analogous to Conjecture~\ref{conj:FJ}. For details and further explanations we recommend~\cite{FRR}, \cite{Kreck-Lueck-Novikov}, \cite{L-book-project}, \cite{LR-survey}, and~\cite{Mislin-Valette}.

In~\cite{FJ-iso}, Farrell and Jones formulated Conjecture~\ref{conj:FJ} not only for algebraic $K$-theory, but also for~$L$-theory; more precisely, for~$\BL^{\langle-\infty\rangle}(R[G])$, the quadratic algebraic $L$-theory spectrum of~$R[G]$ with decoration~$-\infty$, for any ring with involution~$R$.
The corresponding assembly map is constructed completely analogously, by applying the machinery of Subsection~\ref{subsec:assembly} to the functor~$\BL^{\langle-\infty\rangle}(R[-])$.
In the special case of torsion-free groups~$G$, this conjecture is equivalent to the statement that the classical assembly map
\(
BG_+\sma\BL^{\langle-\infty\rangle}(R)\TO\BL^{\langle-\infty\rangle}(R[G])
\)
is a $\pi_*$-isomorphism, for any ring~$R$, not necessarily regular.

If $G$ is a torsion-free group and the Farrell-Jones Conjectures hold for both $\K(\IZ[G])$ and~$\BL^{\langle-\infty\rangle}(\IZ[G])$, then the Borel Conjecture is true for manifolds with fundamental group~$G$ and dimension at least~$5$.
The Borel Conjecture\index{conjecture!Borel} states that, if $M$ and~$N$ are closed connected aspherical manifolds with isomorphic fundamental groups, then $M$ and~$N$ are homeomorphic, and every homotopy equivalence between $M$ and~$N$ is homotopic to a homeomorphism.
In short, the Borel Conjecture says that closed aspherical manifolds are topologically rigid.
Recall that a connected CW~complex~$X$ is aspherical if its universal cover is contractible, or equivalently if $\pi_n(X)=0$ for all~$n>1$.

We also mention that the Farrell-Jones Conjecture in algebraic $L$-theory implies the Novikov Conjecture about the homotopy invariance of higher signatures.

Furthermore, Farrell and Jones also formulated an analog of Conjecture~\ref{conj:FJ} for the stable pseudo-isotopy functor, or equivalently for Waldhausen's $A$-theory, also known as algebraic $K$-theory of spaces.
We refer to~\cite{ELPUW} for a modern approach to this conjecture and in particular for its many applications to automorphisms of manifolds.

Finally, the analog of the Farrell-Jones Conjecture~\ref{conj:FJ} for the complex topological $K$-theory of the reduced complex group $C^*$-algebra of~$G$ is equivalent to the famous Baum-Connes Conjecture, formulated by Paul Baum, Alain Connes, and Nigel Higson in~\cite{Baum-Connes}.
For the Baum-Connes Conjecture, the relative assembly map~$\asbl_{\Fin\subseteq\VCyc}$ is always a $\pi_*$-isomorphism; compare and contrast with Theorem~\ref{thm:Fin-to-VCyc}.
Also the Baum-Connes Conjecture implies the Novikov Conjecture.
For more information on the relation between the Baum-Connes Conjecture and the Farrell-Jones Conjecture in $L$-theory we refer to~\cite{Land-Nikolaus} and~\cite{Rosenberg-for-topologists}.


\section{State of the art}
\label{sec:state}

We now overview what we know and don't know about the Farrell-Jones Conjecture~\ref{conj:FJ}, to the best of our knowledge in January~2017.
We aim to give immediately accessible statements, which may not always reflect the most general available results.
We restrict our attention to algebraic $K$-theory and ignore the related conjectures mentioned in the previous subsection.


\subsection{What we know already}

The following theorem is the result of the effort of many mathematicians over a long period of time.
The methods of controlled algebra and topology that underlie this theorem (and that we illustrate in the next section) were pioneered by Steve Ferry \cite{Ferry} and Frank Quinn \cite{Quinn-Ends}, and were then applied with enourmous success by
Farrell-Hsiang \cite{FH-space-form}, \cite{FH-poly}, \cite{FH-flat}  and  Farrell-Jones \cite{FJ-K-dynamics}, \cite{FJ-Mostow}, \cite{FJ-non-pos}, \cite{FJ-iso}.
Many ideas in the proofs of the following results originate in these articles.
The formulation of the theorem below is meant to be a snapshot of the best results available today, as opposed to a comprehensive historical overview of the many important intermediate results predating the works quoted here.

\begin{theorem}
\label{thm:state}
\index{conjecture!Farrell-Jones!state of the art}
Let $\FJclass$ be the smallest class of groups that satisfies the following two conditions.
\begin{enumerate}[label=(\arabic*)]
\item\label{i:contains}
The class~$\FJclass$ contains:
\begin{enumerate}[label=(\alph*)]
\item \label{i:hyp}
hyperbolic groups \cite{BLR-Inventiones};
\item \label{i:CAT0}
finite-dimensional CAT(0)-groups \cite{BL-Borel}, \cite{Wegner-CAT0};
\item
virtually solvable groups \cite{Farrell-Wu}, \cite{Wegner-solv};
\item\label{i:GMeinertR}
Baumslag-Solitar groups and graphs of abelian groups \cite{Farrell-Wu}, \cite{GMeinertR};
\item\label{i:lattices}
lattices in virtually connected Lie groups \cite{BFL-lattices}, \cite{Kammeyer-etal};
\item \label{arithmetic}
arithmetic and $S$-arithmetic groups \cite{BLRR}, \cite{Rueping-S-arithmetic};
\item
fundamental groups of connected manifolds of dimension at most~$3$ \cite{Roushon-3mfds};
\item \label{i:coxeter}
Coxeter groups;
\item
Artin braid groups \cite{AFR};
\item \label{i:mcg}
mapping class groups of oriented surfaces of finite type \cite{BB-MCG}.
\end{enumerate}

\item\label{i:closed}
The class~$\FJclass$ is closed under:
\begin{enumerate}[label=(\Alph*)]
\item \label{i:sub}
subgroups \cite{BR-coefficients};
\item
overgroups of finite index \cite[Section~6]{BLRR};
\item \label{i:prod}
finite products;
\item \label{i:coprod}
finite coproducts;
\item
directed colimits \cite{BEchterhoffL};
\item
graph products \cite{Gandini-Rueping};
\item
\label{i:closed-under-extensions}
if $1\TO N\TO G\TO[p] Q\TO1$ is a group extension such that $Q\in\FJclass$ and $p^{-1}(C)\in\FJclass$ for each infinite cyclic subgroup~$C\le Q$, then $G\in\FJclass$;
\item \label{i:B-rel-hyp}
if $G$ is a countable group that is relatively hyperbolic to subgroups $P_1,\ldots,P_n$ and each $P_i\in\FJclass$, then $G\in\FJclass$ \cite{B-rel-hyp}.
\end{enumerate}
\end{enumerate}
Then the Farrell-Jones Conjecture~\ref{conj:FJ} holds for any ring~$R$ and for any group~$G\in\FJclass$.
\end{theorem}

\begin{proof}
In order to have the inheritance properties formulated in~\ref{i:closed} one needs to work with a slight generalization of the Farrell-Jones Conjecture.
First, one needs to allow coefficients in arbitrary additive categories with $G$-actions \cite{BR-coefficients};
then, one says that the conjecture with finite wreath products is true for~$G$ if the conjecture holds not only for $G$, but also for all wreath products $G \wr F$ of~$G$ with finite groups~$F$ \cite{Roushon-wreath}, \cite[Section~6]{BLRR}.
The Farrell-Jones Conjecture with coefficients and finite wreath products is true for all groups listed under~\ref{i:contains} and has all the inheritance properties listed under~\ref{i:closed}.

Some of the earlier references given above omit the discussion of the version with finite wreath products; consult \cite[Section~6]{BLRR} and \cite[Proposition~1.1]{Gandini-Rueping} for the corresponding extensions.

We discuss the statements \ref{i:coxeter}, \ref{i:prod}, \ref{i:coprod} and \ref{i:closed-under-extensions}, for which no reference was provided above.
Coxeter groups~\ref{i:coxeter} are known to fall under \ref{i:CAT0} by a result of Moussong; compare \cite[Theorem~12.3.3]{Davis-book}.
For \ref{i:prod} use \cite[Corollary~4.3]{BR-coefficients} applied to the projection to the factors, the Transitivity Principle~\ref{transitivity},
and the fact that the Farrell-Jones Conjecture is known for finite products of virtually cyclic groups.
The extension to the version with finite wreath products uses the fact that $(G_1 \times G_2) \wr F$ is a subgroup of $(G_1 \wr F ) \times (G_2 \wr F)$.
Finite coproducts are treated similarly using property~\ref{i:closed-under-extensions} and the natural map from the coproduct to the product; compare \cite[Proposition~1.1]{Gandini-Rueping}.
Statement~\ref{i:closed-under-extensions} itself is simply a combination of \cite[Corollary~4.3]{BR-coefficients} and the Transitivity Principle~\ref{transitivity}.
\end{proof}


\subsection{What we don't know yet}
\label{subsec:dont-know}

At the time of writing the Farrell-Jones Conjecture~\ref{conj:FJ} seems to be open for the following classes of groups:
\begin{enumerate}[label=(\roman*)]
\index{conjecture!Farrell-Jones!open cases}
\item
Thompson's groups;
\item
outer automorphism groups of free groups;
\item
linear groups;
\item
(elementary) amenable groups;
\item
infinite products of groups (satisfying the Farrell-Jones Conjecture).
\end{enumerate}
However, for some of these groups there are partial injectivity results, as we explain in Remark~\ref{rem:partial-inj} below.


\subsection{Injectivity results}

The next theorem gives two examples of injectivity results for assembly maps in algebraic $K$-theory.
Part~\ref{i:eventual-inj} is proved using the trace methods explained in Section~\ref{sec:trace} below, where more rational injectivity results are described.
Part~\ref{i:KNR} is based on a completely different approach using controlled algebra, the descent method due to Gunnar Carlsson and Erik Pedersen~\cite{CP}.
For this method to work, the group has to satisfy some mild metric conditions, which are not needed for the weaker statement in part~\ref{i:eventual-inj}.
One such condition goes back to \cite{FH-Novikov}.
The condition of {finite asymptotic dimension} appeared in the context of algebraic $K$-theory in~\cite{Bartels-squeezing} and~\cite{CG, CG-Addendum}, and was later generalized to {finite decomposition complexity} in~\cite{RTY}.
The extension to non-classical assembly maps appeared in \cite{BRosenthal, BRosenthal-Erratum} and~\cite{Kasprowski}.
The statement in part~\ref{i:KNR} below is from \cite{KNR} and further improves and
combines these developments.
We also mention \cite{Ferry-Weinberger} for yet another approach to injectivity results.

Recall from Theorem~\ref{thm:Fin-to-VCyc} that the relative assembly map
\[
\pi_n\Bigl(EG(\Fin )_+\sma_{\Or G}\K(\IZ[\oid{G}{-}])\Bigr)
\TO
\pi_n\Bigl(EG(\VCyc)_+\sma_{\Or G}\K(\IZ[\oid{G}{-}])\Bigr)
\]
is always split injective, and it becomes an isomorphism after applying $- \tensor_{\IZ} \IQ$ if $R$ is regular, e.g., if $R=\IZ$.
Therefore the results below would follow if we knew the Farrell-Jones Conjecture~\ref{conj:FJ}.

\begin{theorem}
\label{thm:inj-results}
\index{conjecture!Farrell-Jones!injectivity results}
Assume that there exists a finite-dimensional $EG(\Fin)$, and that there exists an upper bound on the orders of the finite subgroups of~$G$.
\begin{enumerate}[label=(\roman*)]
\item\label{i:eventual-inj}
If $R=\IZ$, then there exists an integer~$L>0$ such that for every~$n\ge L$ the rationalized assembly map
\[
\pi_n\Bigl(EG(\Fin)_+ \sma_{\Or G}\K(\IZ[\oid{G}{-}])\Bigr)\tensor_\IZ\IQ
\TO K_n(\IZ[G])\tensor_\IZ\IQ
\]
is injective.
\item\label{i:KNR}
Assume furthermore that $G$ has regular finite decomposition complexity.
Then for any ring~$R$ the assembly map
\[
EG(\Fin)_+\sma_{\Or G}\K(R[\oid{G}{-}])
\TO
\K(R[G])
\]
is split injective on~$\pi_*$.
\end{enumerate}
\end{theorem}

\pagebreak

\begin{proof}
\ref{i:eventual-inj} is a consequence of Theorem~\ref{thm:lrrv-main} below, or rather of its more general version in~\cite[Main Technical Theorem~1.16]{kc}; see Remark~\ref{rem:lrrv-main}\ref{i:our-B} and~\cite[Theorem~1.15]{kc}, where the result is only stated for cocompact~$EG(\Fin)$, but the proof given on page~1015 only uses finite-dimensionality and the existence of a bound on the order of the finite cyclic subgroups.
\ref{i:KNR} is~\cite[Theorem~1.3]{KNR}.
\end{proof}

\begin{remark} \label{rem:partial-inj}
Theorem~\ref{thm:inj-results} applies to groups for which no isomorphism results were known at the time of writing:
\begin{enumerate}[label=(\roman*)]
\item
The existence of an upper bound on the orders of the finite subgroups of~$G$ follows from the existence of a cocompact~$EG(\Fin)$.
For example, this is the case for outer automorphism groups of free groups, to which Theorem~\ref{thm:inj-results}\ref{i:eventual-inj} then applies.
\item
Regular finite decomposition complexity is a property shared by all groups that are either (a)~of finite asymptotic dimension, (b)~elementary amenable, (c)~linear, or (d)~subgroups of virtually connected Lie groups.
\end{enumerate}
\end{remark}


\section{Controlled algebra methods}
\label{sec:controlled}

As noted in the previous section, most proofs of the Farrell-Jones Conjecture~\ref{conj:FJ} use the ideas and technology of controlled algebra\index{controlled algebra|see {geometric modules}}, which are the focus of this section.
The ultimate goal is to explain the Farrell-Hsiang Criterion for assembly maps to be $\pi_*$-isomorphisms.
The criterion goes back to \cite{FH-space-form} and has been successfully applied in many cases, e.g.~\cite{FH-poly}, \cite{FH-flat}, \cite{Quinn-virtually-abelian}, and plays an important role in the proof of Theorem~\ref{thm:state}\ref{i:contains}\ref{i:lattices}~\cite{BFL-lattices}. 
The formulation that we give here in Theorem~\ref{FH-criterion} is due to~\cite{BL-FH}.

Our goal is to keep the exposition as concrete as possible, and to work out the main details of the proof of the following result, establishing the first nontrivial case of the Farrell-Jones Conjecture.

\begin{theorem} \label{thm:FJ-Zn} The Farrell-Jones Conjecture~\ref{conj:FJ} holds for finitely generated free abelian groups, i.e., for any $n \geq 2$ and for any ring~$R$, the assembly map
\[
E\IZ^n (\Cyc)_+\sma_{\Or \IZ^n}\K(R[\oid{\IZ^n}{-}])
\TO
\K(R[\IZ^n])
\]
is a $\pi_*$-isomorphism.
\end{theorem}

Before we get to the proof, we want to show how Theorem~\ref{thm:FJ-Zn} leads to a simple formula for the Whitehead groups of~$\IZ^n$; the article~\cite{LRosenthal} contains many similar but way more general explicit computations.
The Whitehead groups of~$G$ over~$R$ are defined as $\wh^R_k( G) = \pi_k( \Wh^R(G))$, where $\Wh^R (G) $ is the homotopy cofiber of the classical assembly map $\MOR{\asbl_1}{BG_+\sma\K(R)}{\K(R[G])}$ appearing in Conjecture~\ref{conj:FJ-tor-free-regular}.
Of course, $\wh^\IZ_1(G)=\wh(G)$.

\begin{corollary}
For any $n \geq 2$ and $k \in \IZ$ there are isomorphisms
\[
\wh_k^R( \IZ^n ) \cong \bigoplus_{C \in \MaxCyc  } \wh_k^R( C ) \cong
\bigoplus_{C \in \MaxCyc} NK_k(R) \oplus NK_k(R)
\,,
\]
where $\MaxCyc$ denotes the set of maximal cyclic subgroups of $\IZ^n$.
\end{corollary}

Observe that the set of maximal cyclic subgroups of~$\IZ^n$ can be identified with $\IP^{n-1}( \IQ)$, the set of all $1$-dimensional subspaces of the $\IQ$-vector space~$\IQ^n$.

\begin{proof}
There is a $\IZ^n$-equivariant homotopy pushout square
\[
\begin{tikzcd}
\ds\smallcoprod_{C \in \MaxCyc  } \IZ^n \timesd_C EC
\arrow[r]
\arrow[d]
&
\ds E \IZ^n 
\arrow[d]
\\
\ds\smallcoprod_{C \in \MaxCyc} \IZ^n \timesd_C \pt
\arrow[r]
&
\ds E \IZ^n ( \Cyc)
\mathrlap{\,.}
\end{tikzcd}
\]
Applying $(\,?\,)_+ \sma_{\Or ( \IZ^n )} \K (R[\oid{\IZ^n}{- } ])$ preserves homotopy pushout squares, and the induced left vertical map can be identified with a wedge sum of copies of the classical assembly map~$\asbl_1$ for~$C$, using induction isomorphisms.
The homotopy cofibration sequence
\[
BC_+\sma\K(R)\cong EC_+ \sma_{\Or( C)} \K (R[\oid{C}{ - } ]) \xrightarrow{\asbl_1} \K( R[C] ) \TO \Wh^R( C)
\]
is known to split, and $Wh_n^R (C) \cong NK_n(R) \oplus NK_n(R)$; compare \cite[Section~10]{Swan-higher} and~\cite[Theorem~18.1]{WaldhausenII}.
Therefore we obtain the following homotopy pushout square.
\[
\begin{tikzcd}
\pt
\arrow[r]
\arrow[d]
&
\ds {E \IZ^n}_+ \sma_{\Or\IZ^n} \K ( R[\oid{\IZ^n}{ - } ])
\arrow[d]
\\
\ds\bigvee_{C \in\MaxCyc} \Wh^R( C)
\arrow[r]
&
\ds E \IZ^n ( \Cyc)_+ \sma_{\Or\IZ^n} \K (R[\oid{\IZ^n}{ - } ])
\end{tikzcd}
\]
Theorem~\ref{thm:FJ-Zn} identifies the bottom right corner with~$\K(R[\IZ^n])$, and therefore the homotopy cofiber of the right vertical map agrees with the homotopy cofiber of the classical assembly map for $\IZ^n$, completing the proof.
\end{proof}

Working with the Farrell-Jones Conjecture with coefficients mentioned in the proof of Theorem~\ref{thm:state}, we can use induction and reduce the proof of Theorem~\ref{thm:FJ-Zn} to the case~$n=2$, by applying the inheritance property formulated in Theorem~\ref{thm:state}\ref{i:closed}\ref{i:closed-under-extensions} to  a surjective homomorphism $\IZ^n \TO \IZ^2$.
Notice that for $\IZ^2$ itself Theorem~\ref{thm:state}\ref{i:closed}\ref{i:closed-under-extensions} is useless.

However, even in the case~$n=2$ the full proof of Theorem~\ref{thm:FJ-Zn} involves many technicalities that obscure the underlying ideas.
For this reason, we concentrate on the following partial result.

\begin{proposition} \label{prop:ass-surjective}
The assembly map
\begin{equation}
\label{asbl-K1-Z2}
\pi_1\Bigl( E\IZ^2(\Cyc)_+ \sma_{\Or G} \K(R[\oid{\IZ^2}{-}])\Bigr) \TO
K_1(R[\IZ^2])
\end{equation}
is surjective for any ring~$R$.
\end{proposition}

In the rest of this section we give a complete proof of this proposition modulo Theorem~\ref{thm:small-in-image}, which we use as a black box.
The proof is completed right after the statement of Claim~\ref{Z2-is-FH}.


\pagebreak

\subsection{Geometric modules}

The main characters of controlled algebra are defined next.

\begin{definition}[geometric modules]
\index{geometric modules}
Given a ring~$R$ and $G$-space~$X$, the category~$\CC(X) = \CC^G( X ; R)$ of \emph{geometric $R[G]$-modules over~$X$} is defined as follows.
The objects of~$\CC(X)$ are cofinite free $G$-sets~$S$ together with a $G$-map $\MOR{\varphi}{S}{X}$.
Notice that, given a cofinite free $G$-set~$S$, the $R$-module~$R[S]$ is in a natural way a finitely generated free $R[G]$-module.
The morphisms in~$\CC(X)$ from $\MOR{\varphi}{S}{X}$ to $\MOR{\varphi'}{S'}{X}$ are simply the $R[G]$-linear maps $R[S]\TO R[S']$.
\end{definition}

The category~$\CC(X)$ is additive and depends functorially on~$X$, in the sense that a $G$-map $\MOR{f}{X}{X'}$ induces an additive functor $\MOR{f_*}{\CC(X)}{\CC(X')}$ which sends the object $\varphi$ to $f \circ \varphi$.
Let $\CF(R[G])$ be the category of finitely generated free $R[G]$-modules.
The functor $\MOR{\IU}{\CC(X)}{\CF(R[G])}$ (where $\IU$ stands for underlying) that sends $\MOR{\varphi}{S}{X}$ to~$R[S]$ is obviously an equivalence of additive categories, since $\varphi$ does not enter the definition of the morphisms in~$\CC(X)$.
Therefore we obtain a $\pi_*$-isomorphism
\begin{equation} \label{forget-iso}
\begin{tikzcd}
\K(\CC(X))
\arrow[r, "\IU"', "\simeq"] &
\K(R[G])
\,.
\end{tikzcd}
\end{equation}
However, the advantage of $\CC(X)$ is that morphisms have a geometric shadow in $X$, and if $X$ is equipped with a metric we can talk about their size.

\begin{definition}[support and size]
\index{geometric modules!support and size of morphisms}
Let $\MOR{\alpha}{R[S]}{R[S']}$ be a morphism in~$\CC(X)$ from $\MOR{\varphi}{S}{X}$ to $\MOR{\varphi'}{S'}{X}$. Let $(\alpha_{s's})_{(s',s)\in S'\times S}$ be
the associated matrix.
Define the \emph{support} of~$\alpha$ to be
\[
\supp \alpha=
\SET*{\bigl(\varphi'(s'),\varphi(s)\bigr)\in X\times X}{\alpha_{s's}\neq0}
\subseteq X\times X
\,.
\]
If $X$ is equipped with a $G$-invariant metric~$d$, define the \emph{size} of~$\alpha$ to be
\[
\size \alpha =\sup
\SET*{d\bigl(\varphi'(s'),\varphi(s)\bigr)}{\alpha_{s's}\neq0}
\,.
\]
\end{definition}

\begin{figure}[h]
\centering
\includegraphics{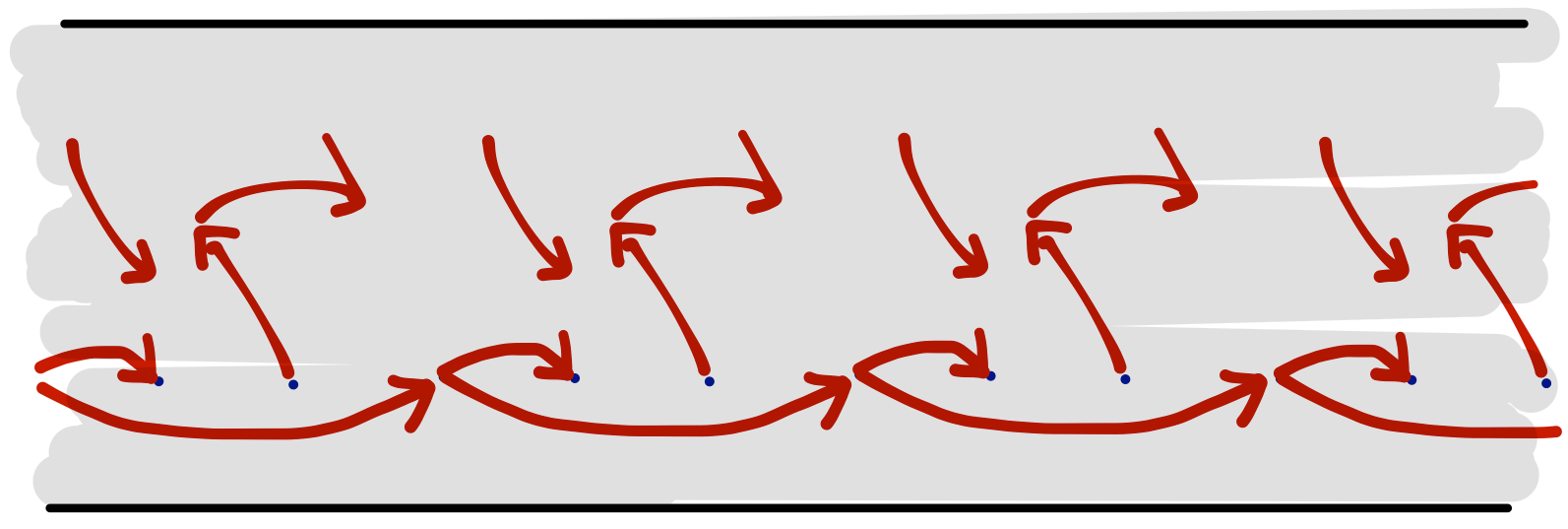}
\caption{Support of a morphism with $G=\IZ$ acting via shift on a band.}
\end{figure}

Note that the supremum is really a maximum, since $\alpha$ is $G$-equivariant, $d$ is $G$-invariant, and $S$ is cofinite.
As we will see, sometimes it is convenient to work with extended metrics, i.e., metrics for which $d(x,x') = \infty$ is allowed.
Being of finite size is then a severe restriction on~$\alpha$. In the support picture no arrow is allowed between points at distance~$\infty$; compare Figure~\ref{pic:ind} on page~\pageref{pic:ind}.

The main idea now is that assembly maps can be described as \emph{forget control} maps. Proving that an element is in the image of an assembly map can be achieved by proving that it has a representative of small size.
Before making this precise we introduce some more conventions and definitions.

Recall that a point $a$ in a simplicial complex $Z$ can be written uniquely in the form
\[
a = \sum_{v \in V} a_v v
\,,
\]
where $V$ is the set of vertices of the underlying abstract simplicial complex, $a_v \in [0,1 ]$, and $\sum_{v \in V} a_v =1$. The point $a$ lies in the interior  of the realization $\Delta_v$ of the unique  abstract simplex given by
$\{ v \; | \; a_v \neq 0 \}$.
The $l^1$-metric on Z is defined as
\[
d^1(a,b) = \sum_{v \in V} |a_v - b_v|
\,.
\]
Observe that the distance between points is always $\leq 2$, and that every simplicial automorphism is an isometry with respect to the $l^1$-metric.

\begin{theorem}[small elements are in the image]
\label{thm:small-in-image}
For any integer $n>0$ there is an $\varepsilon=\varepsilon(n)>0$ such that for every $G$-simplicial complex~$Z$ of dimension $n$ the following is true.
Let $x\in K_1(R[G])$ and consider the assembly map~$\asbl_Z$ induced by $Z\TO\pt$.
\[
\begin{tikzcd}
&
&
K_1(\CC(Z))
\arrow[d, "\cong", "\IU"']
\arrow[r, phantom, "\ni"]
&
\mathclap{[\alpha]}
\arrow[mapsto, d]
\\
\ds\pi_1\Bigl(Z_+\sma_{\Or G}\K(R[\oid{G}{-}])\Bigr)
\arrow[rr, "\asbl_Z"]
&
&
K_1(R[G])
\arrow[r, phantom, "\ni"]
&
\mathclap{x}
\end{tikzcd}
\]
Then $x\in\im(\asbl_Z)$ if there exists an automorphism~$\alpha$ in~$\CC(Z)$ with $\IU([\alpha])=x$ and
\[
\size (\alpha )    \leq\varepsilon
\AND
\size (\alpha^{-1}) \leq\varepsilon
\,.
\]
\end{theorem}

\begin{corollary}
\label{cor:small-in-image}
Retain the notation and assumptions of Theorem~\ref{thm:small-in-image}.
If all isotropy groups of~$Z$ belong to the family~$\CF$, then $x$ is also in the image of the assembly map
\begin{equation} \label{K1-F-asbl}
\pi_1\Bigl(EG(\CF)\sma_{\Or G}\K(R[\oid{G}{-}])\Bigr)
\xrightarrow{\asbl_{\CF}}
K_1(R[G])
\,.
\end{equation}
\end{corollary}

\begin{proof}
The universal property of $EG( \CF)$ in Remark~\ref{rem:source-assembly}\ref{rem-i:rel-assembly} gives a $G$-equivariant map $Z \TO EG(\CF)$.
Hence the assembly map~$\asbl_Z$, which is induced by $Z \TO \pt$, factors over the assembly map~$\asbl_\CF$, which is induced by $EG( \CF) \TO \pt$.
\end{proof}

The sufficient condition for surjectivity on~$\pi_1$ from the preceding two results is generalized in Theorem~\ref{obstruction-category} below to a necessary and sufficient condition for assembly maps to be $\pi_*$-isomorphisms.
In Remark~\ref{rem:small-in-image} we explain how and where in the literature Theorem~\ref{thm:small-in-image} is proved.


\subsection{Contracting maps}

In view of Theorem~\ref{thm:small-in-image} and Corollary~\ref{cor:small-in-image}, a possible strategy to prove surjectivity of~$\asbl_{\CF}$ is to look for contracting maps.
This leads to the following criterion.

\begin{criterion} \label{crit:criterion}
Fix $G$, $R$, $\CF$, and a word metric $d^G$ for $G$.
Suppose that there is an~$N >0$ such that for any arbitrarily large $D >0$ there exist a simplicial complex~$Z_D$ with a simplicial $G$-action and a $G$-equivariant map
\(
\MOR{f_D}{ G/1}{ Z_D}
\)
satisfying the following conditions:
\pagebreak
\begin{enumerate}[label=(\roman*)]
\item \label{bound-dim} $\dim Z_D \leq N$;
\item \label{isotropy} all isotropy groups of $Z_D$ lie in $\CF$;
\item \label{contract} the map $f_D$ is $D$-contracting with respect to the $l^1$-metric in the target and the word metric in the source, i.e., for all~$g,g'\in G$ we have
\[
d^1 ( f_D(g) , f_D(g') ) \leq \frac{1}{D} d^G ( g,g')
\,.
\]
\end{enumerate}
Then the map~\eqref{K1-F-asbl} is surjective.
\end{criterion}

The projection map to a point always satisfies \ref{bound-dim} and \ref{contract} but not \ref{isotropy}. The $N$-skeleton of a simplicial model for $EG( \CF)$ always satisfies \ref{bound-dim} and \ref{isotropy}.
But how can we produce contracting maps $f_D$ that satisfy all three conditions?
In Remark~\ref{impossible} below we explain why the assumptions of the criterion are too strong to be useful.
Nevertheless, we spell out the proof of the criterion as a warm-up exercise.

\begin{proof}
Set $\epsilon = \min \SET{\epsilon(n)}{n \leq N}$, where $\epsilon (n)$ comes from Theorem~\ref{thm:small-in-image}.
Given any~$x\in K_1( R[G])$ consider the following diagram.
\[
\begin{tikzcd}
\ds K_1( \CC( Z_D ) )
\arrow[dr, "\IU\ "', "\cong" description]
&
K_1( \CC( G/1) )
\arrow[d, "\ \IU", "\cong" description]
\arrow[l,"{f_D}_*"']
\arrow[r, phantom, "\ni" pos=.4]
&
\mathclap{[\alpha]}
\arrow[mapsto, d]
\\
&
K_1(R[G])
\arrow[r, phantom, "\ni"]
&
\mathclap{x}
\end{tikzcd}
\]
Choose an automorphism $\alpha$ in $\CC( G/1)$ whose class $[\alpha] \in K_1( \CC( G/1))$ maps to $x$ under~$\IU$. Determine the sizes of $\alpha$ and $\alpha^{-1}$, and then choose $D$ so large that
\[
\size {f_D}_* ( \alpha ) \leq \frac{1}{D} \size  \alpha  < \epsilon
\]
and analogously for $\alpha^{-1}$.
Then Corollary~\ref{cor:small-in-image} implies that $x$ is in the image of the assembly map~$\asbl_{\CF}$ in~\eqref{K1-F-asbl}.
\end{proof}

\begin{remark} \label{impossible}
The case of Proposition~\ref{prop:ass-surjective} is when $G = \IZ^2$ and $\CF = \Cyc$.
Unfortunately, the conditions of Criterion~\ref{crit:criterion} cannot possibly be satisfied in this case.
To explain why, we need the following lemma.

\begin{lemma}
Let $s$ be a simplicial automorphism of a simplicial complex $Z$ with $\dim Z \leq N$. If $x = \sum x_v v \in Z$ is such that
\[
d^1( x , sx) < \frac{1}{(N+1)^2}
\,,
\]
then a barycenter of a face of the simplex $\Delta_x$ spanned by $V_x=\SET{v}{x_v\neq0}$ is fixed under $s$.
\end{lemma}

\begin{proof}
For a vertex $v$ with $x_v \neq 0$ set $A_v = \{\, v,sv,s^2v, \dotsc \} $.
If $A_v \subset V_x$ then $s$ permutes the finitely many elements in $A_v$ and in particular fixes the barycenter of the face spanned by $A_v$.

Suppose that for no vertex $v$ with $x_v \neq 0$ we have $A_v \subset V_x$. Then for all $v \in V_x$ there exists a smallest $n(v) \geq 1$ such that
$s^{n(v)}v \notin V_x$ and hence $x_{s^{n(v)}v} =0$. Since $V_x$ contains at most $N+1$ vertices we know that $n(v) \leq N+1$ for all $v \in V_x$.
Write $\epsilon = \frac{1}{(N+1)^2}$; then from
\[
d^1(x,s^{-1}x) = d^1(s^{-1}x,s^{-2}x) = \dotsb = d^1(s^{-N}x, s^{-(N+1)}x) < \epsilon
\]
and $x_{s^{n(v)}v}=0$ we conclude that
\[
x_{s^{n(v)-1} v} < \epsilon\,, \quad  x_{s^{n(v)-2} v} < 2 \epsilon\,, \quad \dotsc, \quad   x_v < n(v) \epsilon \leq (N+1) \epsilon = \frac{1}{N+1}
\,.
\]
However, since $\sum_{v \in V_x} x_v = 1$, the last inequality cannot be true for all vertices in~$V_x$.
\end{proof}

Now suppose we can arrange \ref{bound-dim} and \ref{contract} from Criterion~\ref{crit:criterion}.
If $S$ is a (very large) finite subset of $G$, then by \ref{contract} there exists a $G$-equivariant map $\MOR{f}{G/1}{Z}$ to a $G$-simplicial complex that is  contracting enough in order to have $d^1( f(1) , f(s) ) < \frac{1}{(N+1)^2}$ for all $s \in S$. The lemma implies that for each $s \in S$ a barycenter of a face of the simplex $\Delta_{f(1)}$ determined
by $f(1)$ is fixed under $s$.
Let $b(N)$ be the number of vertices in the barycentric subdivision of an $N$-simplex.
Then there exists a subset $T \subset S$ with cardinality $\# T \geq \frac{ \# S}{b(N)}$ and a point in $\Delta_{f(1)}$ fixed by all elements of~$T$.
The subgroup generated by $T$ must lie in $\CF$ if we require \ref{isotropy}. Since $S$ can be arbitrarily large it seems difficult to keep $\CF$ small.

In the case $G = \IZ^2$ we can choose
\[
S =S_l = \SET*{(x_1 , x_2)}{|x_1| \leq \frac{l}{2} \text{ and } |x_2| \leq \frac{l}{2}}
\,.
\]
Then if $l > b(N)$ we have $\#T \geq \frac{l^2}{b(N)} > l $, and since a subset of $S_l$ with more than $l$ elements cannot be contained in a line, the set $T$ generates a finite index subgroup.
Hence we can never arrange $\CF = \Cyc$ as desired, proving the claim in Remark~\ref{impossible}.
\end{remark}


\subsection{The Farrell-Hsiang Criterion}

The trick to obtain sufficiently contracting maps is to relax the requirement that the maps are $G$-equivariant, and instead only ask for equivariance with respect to (finite index) subgroups. We first illustrate this phenomenon in an example that is too simple to be useful.

\begin{example}
Consider the standard shift action of the infinite cyclic group $G= \IZ$ on the real line:
\[
\IZ \times \IR \TO \IR\,, \quad (z,x) \longmapsto z + x\,.
\]
This is a simplicial action if we consider $\IR$ as $1$-dimensional simplicial complex with set of vertices $\IZ \subset \IR$. The map
\[
f_D \colon \IZ \TO \IR\,, \quad z \longmapsto \frac{1}{D} z
\]
is $D$-contracting but not $\IZ$-equivariant. It becomes $\IZ$-equivariant if we change the action on~$\IR$ to the action given by
\[
\IZ \times \IR \TO \IR\,, \quad (z,x) \longmapsto \frac{1}{D}z + x\,.
\]
However, this action is no longer simplicial. If we restrict the action to the subgroup $D \IZ < \IZ$ or to any subgroup $H$ with $H \le D \IZ$, then the $H$-action on~$\res_H^{\IZ} \IR$ is simplicial, and
\[
f_D \colon \res_H^{\IZ} \IZ \TO \res_H^{\IZ} \IR
\]
is a $D$-contracting $H$-equivariant map.
\end{example}

Assume for a moment that for a subgroup $H \le G$ of finite index we have  an $H$-equivariant map
\[
f_D \colon \res_H^G G/1 \TO E_H
\]
to an $H$-simplicial complex~$E_H$ that is $D$-contracting with respect to a word metric in the source and the $l^1$-metric in the target. Let us see what happens when we induce up to~$G$.

If $(X,d)$ is a metric space with an isometric $H$-action, then $\ind_H^G X = G \times_H X$ has an isometric $G$-action with respect to the extended metric
\[
d( [g,x] , [g',x'] ) =
\begin{cases}
d_X ( x , g^{-1} g' x) & \text{ if\ \ $g^{-1} g' \in H$;}
\\
\infty & \text{ if\ \ $g^{-1} g' \not\in H$.}
\end{cases}
\]
Applying this to $f_D$ we obtain a map
\[
\ind_H^G f_D \colon \ind_H^G  \res_H^G G/1 \TO \ind_H^G   E_H
\]
which is still $D$-contracting.
However, observe that $\frac{1}{D} \infty = \infty$, and that a pair of points at distance~$\infty$ in the source is mapped to a pair of points still at distance~$\infty$ in the target. Hence the map can be used to diminish the size of a morphism between geometric modules only if the morphism over $\ind_H^G\res_H^G G/1$ is of finite size, i.e., only if it has no components that connect points at distance~$\infty$.

The usual induction homomorphism $\ind_H^G \colon K_1( R[H] ) \TO K_1(R[G])$ given by the functor $R[G] \tensor_{R[H]} - $ can be easily lifted to  the categories of geometric modules, i.e., for any metric space $X$ the functor
\[
\ind_H^G \colon \CC ( X) \TO \CC( \ind_H^G X )\,, \quad ( \varphi \colon S \TO X ) \longmapsto ( \ind_H^G \varphi \colon \ind_H^G S \TO \ind_H^G X )
\]
induces the upper horizontal map in the following commutative diagram.
\[
\begin{tikzcd}
\ds K_1( \CC ( X ))
\arrow[r, "\ind_H^G"]
\arrow[d, "\cong", "\IU"']
&
K_1( \CC ( \ind_H^G X ))
\arrow[d, "\cong", "\IU"']
\\
K_1 ( R[H] )
\arrow[r, "\ind_H^G"]
&
K_1(R[G])
\end{tikzcd}
\]
If $X$ is a metric space in the usual sense (where $\infty$ is not allowed), then morphisms in the image of $\ind_H^G$ have the desired property: the size of $\ind_H^G \alpha$ is finite even though $\ind_H^G X$ is a metric space in the extended sense.
Moreover
\[
\size \ind_H^G \alpha = \size \alpha\,.
\]
Therefore, using the map $\ind_H^G f_D$ we can hope to show that, maybe not arbitrary elements, but at least elements of the form $\ind_H^G [\beta]$ belong to the image of $\asbl_{\CF}$.

\begin{figure}[h]
\centering
\includegraphics[width=.75\textwidth]{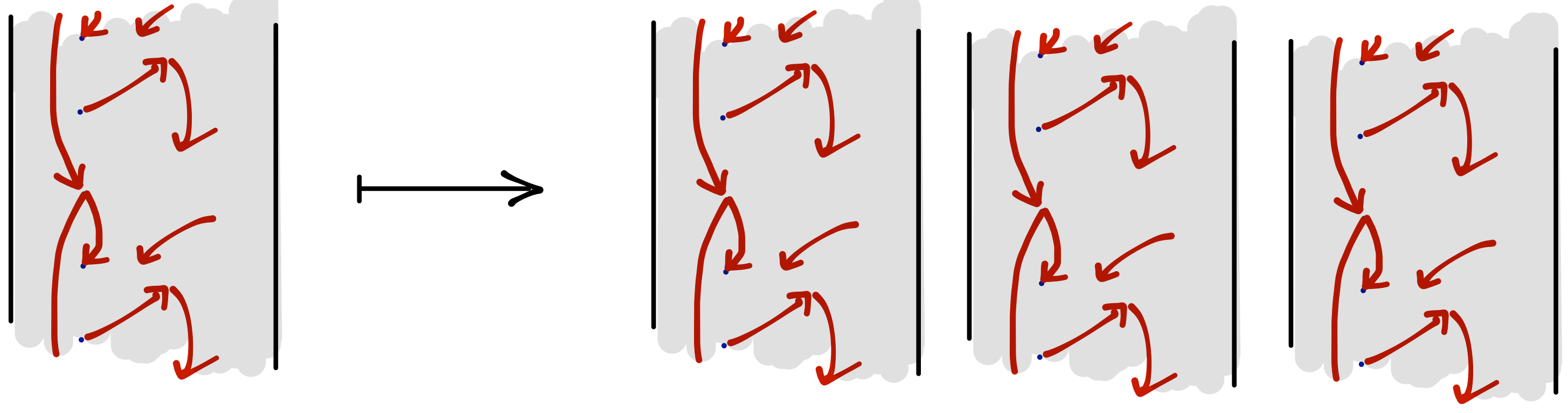}
\caption{Support of $\alpha$  and $\ind_H^G \alpha$ in an index $3$ situation.}
\label{pic:ind}
\end{figure}

The reason why this is useful is the following theorem of Swan. Recall that a finite group~$E$ is called hyperelementary if it fits into a short exact sequence
\(
1\TO C\TO E\TO P\TO 1
\)
where $C$ is cyclic and the order of $P$ is a prime power.

\begin{theorem}[Swan induction] \label{thm:swan}
Let $F$ be a finite group, $\MOR{\pr}{G}{F}$ a surjective homomorphism, and
\[
\CH_{\pr}=\SET*{\pr^{-1}(E)}{\text{$E$ is a hyperelementary subgroup of~$F$}}
\,.
\]
Then for every $H\in\CH_{\pr}$ there exist $\IZ [H]$-modules $M_H^+$ and $M_H^-$ that are finitely generated free as $\IZ$-modules, and such that,
for each $n\in\IZ$ and each $x\in K_n(R[G])$, we have
\begin{equation} \label{swan-formula}
x=\sum_{H\in\CH_p}  \ind_H^G \left( [M_H^+] \cdot \res_H^G x \right) - \ind_H^G \left( [M_H^-] \cdot \res_H^G x \right)
.
\end{equation}
\end{theorem}
Here, for $y \in K_n(R[H])$ and a $\IZ [H]$-module $M$ which is finitely generated free as a $\IZ$-module, we write $[M]\cdot y = l_M ( y)$ for the image of $y$ under the map induced in $K$-theory by the functor $l_{M}$ that sends the $R [H]$-module $P$ to $M \tensor_{\IZ} P$ equipped with the diagonal $H$-action.
\begin{proof}
The Swan group $Sw( H ; \IZ)$ is by definition the $K_0$-group of $\IZ H$-modules that are finitely generated free as $\IZ$-modules. The relation is the usual additivity relation for (not necessarily split) short exact sequences.
Tensor products over $\IZ$ equipped with the diagonal $H$-actions induce the structure of a unital commutative ring on $Sw( H; \IZ)$, and also define an action of $Sw ( H; \IZ)$ on $K_n ( R[H])$. Swan showed in~\cite{Swan}
that for a finite group $F$ there exist $\IZ[E]$-modules $N_E^+$ and $N_E^-$, where $E$ runs through all hyperelementary subgroups of $F$, such that in $Sw( F ; \IZ)$ we have
\begin{equation} \label{swan-formula-original}
1 = [ \IZ ] = \sum_{E}  \ind_E^F [N_E^+]  - \ind_E^F [N_E^-]
\,.
\end{equation}
The natural isomorphisms
\[
\ind_H^G\Bigl( M \tensor_\IZ \res_H^G P\Bigr) \stackrel{\cong}{\to} \bigl(\ind_H^G M \bigr)\tensor_{\IZ} P
\aND
\ind_{\pr^{-1}(E)}^G  \res_{\pr} N \stackrel{\cong}{\to}  \res_{\pr} \ind_E^F N
\,,
\]
given by $g \otimes m \otimes p \mapsto g \otimes m \otimes gp$ and $g \otimes n \mapsto \pr(g) \otimes n$, respectively, yield the following identity in $K_n( R[G] )$ for $H = \pr^{-1}(E)$:
\[
\bigl( \res_{\pr} \ind_E^F [N]\bigr) \cdot x = \bigl( \ind_H^G [\res_{\pr} N] \bigr) \cdot x= \ind_H^G \bigl( [\res_{\pr} N] \cdot \res_H^G  x\bigr)
\,.
\]
Using this and  $\res_{\pr} 1 = [ \res_{\pr} \IZ ] = [\IZ] = 1$ one derives the statement in the theorem with $M_{\pr^{-1}(E)} = \res_{\pr} N_E$ from \eqref{swan-formula-original}.
\end{proof}

If we want to use $H$-equivariant contracting maps, as explained above, to show that each of the summands in~\eqref{swan-formula} is in the image of $\asbl_{\CF}$, we need to control the size of a geometric representative of $[M_H^{\pm}] \cdot \res_H^G x $ in terms of the size of a representative of $x$.

This is indeed easy. Similarly to induction, also the functors restriction $\res_H^G$ and $l_{M} = M \tensor_{\IZ} - $  can be lifted to categories of geometric modules. For restriction simply send the object given by $\phi \colon S \TO Z$ to $\res_H^G \phi \colon \res_H^G S \TO \res_H^G Z$. For~$l_{M}$ observe that if $B$ is a finite $\IZ$-basis for the $\IZ[H]$-module $M$, then there are isomorphisms of $\IZ[H]$-modules
\begin{equation} \label{diag-vs-right}
\IZ[ B ] \tensor_{\IZ} R[ \smallcoprod H/1 ] \TO[\cong] \IZ[ B] \tensor_{\IZ} R[ \smallcoprod H/1 ] \TO[\cong]  R[B \times \smallcoprod H/1 ]
\,.
\end{equation}
Here the first isomorphism is given by $m \otimes h \mapsto h^{-1}m \otimes h$, where in the source one uses the diagonal $H$-action, and in the target the $H$-action on the right tensor factor. The second isomorphism is the obvious one. One constructs the desired functor by working only with objects of the form $\phi \colon \coprod H/1 \TO Z$ and sending such a $\phi$ to $\phi \circ \pr$, where $\pr \colon B \times \coprod H/1 \TO \coprod H/1$ is the projection onto the second factor. The behaviour on morphisms is determined by the isomorphism~\eqref{diag-vs-right}: one defines $l_M \alpha$ between the objects on the right in~\eqref{diag-vs-right} in such a way that on the left it corresponds to~$\id \otimes \alpha$. One then checks easily that
\[
\size \res_H^G \alpha  = \size \alpha \quad \mbox{ and } \quad \size l_{M} \alpha =  \size \alpha
\,.
\]
In summary, given a finite index subgroup $H \le G$, a $\IZ[H]$-module $M$ that is finitely generated free as a $\IZ$-module, and an $H$-equivariant $D$-contracting map $f_D \colon \res_H^G G/1 \TO Z$ to an $H$-simplicial complex, we have a commutative diagram
\begin{equation}
\label{huge-diagram}
\adjustbox{width=\textwidth,keepaspectratio}%
{\begin{tikzcd}
{[ \alpha ]}
\arrow[d, phantom, "\in"]
\arrow[mapsto, rr, shorten >= 6.3em]
&
&
\mathclap{\left[ (\ind_H^G f_D)_* ( \ind_H^G l_M \res_H^G \alpha ) \right]}
\arrow[r, phantom, "\in" pos=.825]
&
K_1( \CC ( \ind_H^G Z))
\\
K_1( \CC( G/1) )
\arrow[r, "\res_H^G"]
\arrow[d, "\cong", "\IU"']
&
K_1( \CC ( \res^G_H G/1 ))
\arrow[r, "l_M"]
\arrow[d, "\cong", "\IU"']
&
K_1( \CC( \res_H^G G/1 ))
\arrow[r, "\ind_H^G"]
\arrow[d, "\cong", "\IU"']
&
K_1( \CC ( \ind_H^G \res_H^G G/1 ))
\arrow[d, "\cong", "\IU"']
\arrow[u, "(\ind_H^G f_D)_*"]
\\
K_1( R[G] )
\arrow[r, "\res_H^G"]
&
K_1( R[H] )
\arrow[r, "l_M"]
&
K_1( R[H] )
\arrow[r, "\ind_H^G"]
&
K_1( R[G] )
\end{tikzcd}}
\end{equation}
and the estimate
\begin{equation}\label{estimate-huge}
\size \left( (\ind_H^G f_D)_* ( \ind_H^G l_M \res_H^G \alpha ) \right) \leq \frac{1}{D} \size
\left( \ind_H^G l_M \res_H^G \alpha \right) = \frac{1}{D}  \size \alpha < \infty
\,.
\end{equation}

In order to prove surjectivity of $\asbl_{\CF}$ it  remains to find suitable finite quotients $\pr \colon G \TO F$ and  suitable $H$-equivariant contracting maps for each~$H \in \CH_{\pr}$.
This leads to the criterion formulated in Theorem~\ref{FH-criterion} below for arbitrary groups~$G$.
Groups~that meet this criterion have been named Farrell-Hsiang groups in~\cite{BL-FH}.


\subsection{\texorpdfstring{$\IZ^2$}{ZxZ} is a Farrell-Hsiang group}

Now we concentrate on the concrete situation where $G = \IZ^2$, and explain how the criterion is met in this special case.

\begin{claim}[$\IZ^2$ is a Farrell-Hsiang group with respect to $\Cyc$] \label{Z2-is-FH}
\index{theorem!Farrell-Hsiang Criterion!special case of Z2@special case of~$\IZ^2$}
Fix a word metric $d^{\IZ \times \IZ}$ on~$\IZ\times\IZ$. Consider~$\IR$ as a simplicial complex with vertices~$\IZ\subset\IR$ and with the corresponding $\ell^1$-metric~$d^1$.
For any arbitrarily large~$D>0$ there exists a surjective homomorphism $\MOR{\pr_D}{\IZ\times\IZ}{F}$ to a finite group~$F$ with the following property.
For each
\[
H\in\CH_{\pr_D}=\SET*{\pr_D^{-1}(E)}{\text{$E$ is a hyperelementary subgroup of $F$}}
\]
there exist:
\begin{enumerate}[label=(\roman*)]
\item
a simplicial $H$-action on~$\IR$ with only cyclic isotropy,
\item
a map
\(
\MOR{f_H}{\res_H (\IZ\times\IZ)}{\IR}
\)
that is $H$-equivariant and $D$-contracting, i.e.,
\begin{equation}
\label{eq:contracting}
d^1(f_H(g),f_H(g'))
\leq
\tfrac{1}{D}
\,
d^{\IZ\times\IZ}(g,g')
\end{equation}
for all $g,g'\in\IZ\times\IZ$.
\end{enumerate}
\end{claim}

We first show that this implies Proposition~\ref{prop:ass-surjective}.

\begin{proof}[Proof of Proposition~\ref{prop:ass-surjective}]
The simplicial complex $\IR$ is $1$-dimensional.
Let $\epsilon = \epsilon (1)$ be as in Theorem~\ref{thm:small-in-image}. Given $x \in K_1( R[G])$ choose an automorphism $\alpha$ in $\CC( G/1)$ such that $[\alpha]$ maps to $x$ under the forgetful map $\IU \colon K_1( \CC( G/1)) \TO K_1( R[G])$.
Choose $D >0$ so large that $\frac{1}{D} \max \{\size (\alpha), \size (\alpha^{-1}) \} \leq \epsilon$.
Use Claim~\ref{Z2-is-FH} in order to find a finite quotient $\pr_D \colon \IZ \times \IZ \TO F$ and $H$-equivariant $D$-contracting maps $f_H \colon \res_H^G G/1 \TO \IR$ for every $H \in \CH_{\pr_D}$.

For each $H \in \CH_{\pr_D}$, let $M=M_{H}^{\pm}$ be as in Theorem~\ref{thm:swan}, and send $[ \alpha ]$ through the upper row in diagram \eqref{huge-diagram}.
Use estimate~\eqref{estimate-huge} to conclude that
\[
\size \left( (\ind_H^G f_H)_*(\ind_H^G l_M \res_H^G \alpha) \right) \leq \epsilon
\,.
\]
By Corollary~\ref{cor:small-in-image} and the commutativity of \eqref{huge-diagram}, we see that $\ind_H^G l_M \res_H^G x$ is in the image of the map~\eqref{asbl-K1-Z2}.
Because of the decomposition~\eqref{swan-formula} in Theorem~\ref{thm:swan}, also $x$ is in the image.
\end{proof}

\begin{proof}[Proof of Claim~\ref{Z2-is-FH}]
We begin with some simplifications.
With respect to the standard generating set~$\{(\pm1,0),(0,\pm1)\}$, the word metric is Lipschitz equivalent to the Euclidean metric on~$\IZ\times\IZ\subset\IR\times\IR$.
On~$\IR$ the simplicial $\ell^1$ metric and the Euclidean metric satisfy
\[
d^1(x,y)\le C\,d^\Eucl(x,y)
\]
for some fixed constant~$C$.
Therefore it is enough to establish~\eqref{eq:contracting} with respect to the Euclidean metrics on~$\IZ\times\IZ\subset\IR\times\IR$ and on~$\IR$, instead of the word and $\ell^1$-metrics.
Moreover, it is enough to consider only maximal hyperelementary subgroups of~$F$, because then for any $H'<H$ we can take $f_{H'}=\res_{H'}f_H$.

Let us start to look for suitable finite quotients~$F$ of~$\IZ\times\IZ$.
If $F$ itself were hyperelementary, then we would have to find a contracting map~$f_{\IZ\times\IZ}$ to a $(\IZ\times\IZ)$-simplicial complex with cyclic isotropy that is $(\IZ\times\IZ)$-equivariant.
But in Remark~\ref{impossible} we saw that this is impossible.

Every finite quotient~$F$ of~$\IZ\times\IZ$ is isomorphic to~$\IZ/a\times\IZ/ab$, which is hyperelementary if and only if $a$ is a prime power.
Hence a simple choice of~$F$ which is not itself hyperelementary is $\IZ/pq\times\IZ/pq$ for distinct primes $p$ and~$q$.
In order to achieve the contracting property we will later choose the primes to be very large.

Let $\MOR{\pr_{pq}}{\IZ\times\IZ}{\IZ/pq\times\IZ/pq}$ be the projection.
A maximal hyperelementary subgroup~$E$ of~$\IZ/pq\times\IZ/pq$ has order $pq^2$ or~$p^2q$.
By symmetry it is enough to consider the case where the order of~$E$ is~$pq^2$.
Let
\(
H = \pr_{pq}^{-1}(E)
\).
Now we need to construct~$f_H$.

For every $v\in\IZ\times\IZ$ with $v\neq0$, consider the map
\[
f_v\colon
\IZ\times\IZ
\xrightarrow{\ell_v=\langle v,-\rangle}
\IZ
\xrightarrow{\quad-/p\quad}
\IR
\,, \quad w \longmapsto \frac{1}{p} \langle v,w \rangle
\]
where $\langle-,-\rangle$ is the standard inner product on~$\IR^2$.
If we equip~$\IR$ with the $(\IZ\times\IZ)$-action given by
\[
(\IZ\times\IZ)\times\IR\TO\IR
\,,\qquad
(w,x)\longmapsto x+\tfrac{1}{p}\langle v,w\rangle
\]
then $f_v$ is $(\IZ\times\IZ)$-equivariant.
More importantly, we have that:
\begin{enumerate}[label=(\Alph*)]
\item \label{i:contracting}
$f_v$ is $p/\|v\|$-contracting, i.e.,
\(
|f_v(w)-f_v(w')|\leq\tfrac{\|v\|}{p}\|w-w'\|
\).
This follows immediately from the linearity of~$f_v$ and the Cauchy-Schwarz inequality.

\item \label{i:cyclic-isotropy}
The isotropy group at every point of~$\IR$ is $\ker(\ell_v)=\SET{w\in\IZ\times\IZ}{\langle v,w\rangle=0}$, and hence cyclic since we assumed that~$v\neq0$.

\item \label{i:simplicial}
The action restricts to a simplicial $H$-action if $\ell_v(H)\subseteq p\IZ$.
\end{enumerate}
Let us reformulate the last condition.
Consider the following commutative diagram.
\begin{equation}
\label{eq:H-E-diagram}
\begin{tikzcd}
H=\pr_{pq}^{-1}(E)
\arrow[r, hook]
\arrow[d, two heads]
&
\IZ\times\IZ
\arrow[rr, "\ell_v"]
\arrow[d, two heads, "\pr_{pq}"]
&&
\IZ
\arrow[dd, two heads]
\\
E
\arrow[r, hook]
\arrow[d, two heads]
&
\IZ/pq\times\IZ/pq
\arrow[d, two heads, "\pr"]
\\
\pr(E)=\IF_p\!\cdot\overline{u}
\arrow[r, hook]
&
\IF_p\times\IF_p
\arrow[rr, "\overline{\ell_v}=\ell_{\overline{v}}"']
&&
\IF_p
\end{tikzcd}
\end{equation}
Here $\overline{u}\in\IF_p\times\IF_p$ is a generator of the $\IF_p$-vector space~$\pr(E)<\IF_p\times\IF_p$.
Observe that $\pr(E)\neq\IF_p\times\IF_p$ because the order of~$E$ is~$pq^2$.
Then the last condition above is equivalent to saying that the composition in diagram~\eqref{eq:H-E-diagram} from~$H$ to~$\IF_p$ is trivial, i.e., that $\ell_{\overline{v}}(\overline{u})=0$.

Hence, if we can find a vector $v\in\IZ\times\IZ$ such that
\begin{equation}
\label{eq:conditions-on-v}
0<\|v\|\leq4\sqrt{p}
\AND
\ell_{\overline{v}}(\overline{u})=0
\,,
\end{equation}
then from \ref{i:contracting} we get that $f_v$ is a $(\,p/4\sqrt{p}=\!\sqrt{p}/4\,)$-contracting $H$-equivariant map to~$\IR$, where~$\IR$ is equipped with a simplicial $H$-action  by~\ref{i:simplicial} and has cyclic isotropy by~\ref{i:cyclic-isotropy}.

The existence of such a vector~$v$ is established by the following counting argument.
Consider the set
\[
S=\SET*{v=(x_1,x_2)\in\IZ\times\IZ}{|x_1|\leq\sqrt{2p}\text{ and }|x_2|\leq\sqrt{2p}}
.
\]
This set has more than~$p$ elements, and therefore the map
\[
S\TO\IF_p
\,,\qquad
v\longmapsto\ell_{\overline{v}}(\overline{u})
\]
is not injective, where $\overline{u}$ was defined right after diagram~\eqref{eq:H-E-diagram}.
If $v_0$ and~$v_1$ are two distinct vectors in~$S$ with
\(
\ell_{\overline{v_0}}(\overline{u})
=
\ell_{\overline{v_1}}(\overline{u})
\),
then $v=v_0-v_1$ is a vector which satisfies the equality in~\eqref{eq:conditions-on-v}.
For the inequality in~\eqref{eq:conditions-on-v}  we estimate
\[
\|v\|\leq\|v_0\|+\|v_1\|\leq2\sqrt{2}\sqrt{2p}=4\sqrt{p}
\,.
\]
So we define~$f_H=f_v$ for such a~$v$ and finish the argument using Euclid's Theorem: since there are infinitely many primes, for any given~$D>0$ we can find distinct primes $p$ and~$q$ such that both
\(
\sqrt{p}/4\geq D
\)
and
\(
\sqrt{q}/4\geq D
\),
and hence for every
\[
H\in\CH_{\pr_{pq}}=\SET*{\pr_{pq}^{-1}(E)}{\text{hyperelementary $E<\IZ/pq\times\IZ/pq$}}
\]
the map $f_H$ is $D$-contracting.
\end{proof}


\subsection{The Farrell-Hsiang Criterion (continued)}

We now indicate how the ideas developed in this section can be used to prove isomorphism results in all dimensions instead of just surjectivity results for $K_1$.
In~\cite{BLR-Inventiones} the authors introduce, for an arbitrary  $G$-space $X$, the additive categories $\CT^G(X)$, $\CO^G(X)$, and $\CD^G(X)$, and establish in \cite[Lemma~3.6]{BLR-Inventiones} a homotopy fibration sequence
\begin{equation} \label{eq:TOD}
\K( \CT^G(X) ) \TO
\K( \CO^G(X) ) \TO
\K( \CD^G(X) )
\,.
\end{equation}
The category $\CT^G( X)$ is a variant of the category denoted $\CC ( X)$ in this section.
The functor $X \longmapsto \K( \CD^G( X) )$ is a $G$-equivariant homology theory on $G$-CW~complexes \cite[Section~5]{BFJR}, and the value at $G/H$ is $\pi_*$-isomorphic to $\Sigma \K( R[H] )$ \cite[Section~6]{BFJR}. Therefore the general principles in \cite{Weiss-Williams} and~\cite{Davis-L} identify the map
\[
\K ( \CD^G( EG( \CF) ) ) \TO \K ( \CD^G( \pt ) )
\]
with the (suspended) assembly map $\asbl_{\CF}$.

A variant of the category $\CO^G( X)$ can be defined as follows.
Objects are $G$-equivariant maps $\varphi \colon S \TO X \times [1, \infty)$, where now the free $G$-set $S$ is allowed to be cocountable instead of only cofinite.
Moreover we require that
$\varphi^{-1}( X \times [1,N] )$ is cofinite for every $N$.

A morphism $\alpha$ from $\varphi$ to $\varphi^{\prime}$ is again an $R[G]$-linear map $\alpha \colon R[S] \TO R[S']$,
but now there is a severe restriction on the support of a morphism: towards $\infty$ the arrows representing non-vanishing components must become smaller and smaller. Notice though that $X$ is only a topological and  not a metric space, and ``small'' has no immediate meaning. We refer to \cite[Definition~2.7]{BFJR} 
for the precise definition of this condition, which is known as \emph{equivariant continous control at infinity}.


\begin{figure}[h]
\centering
\includegraphics{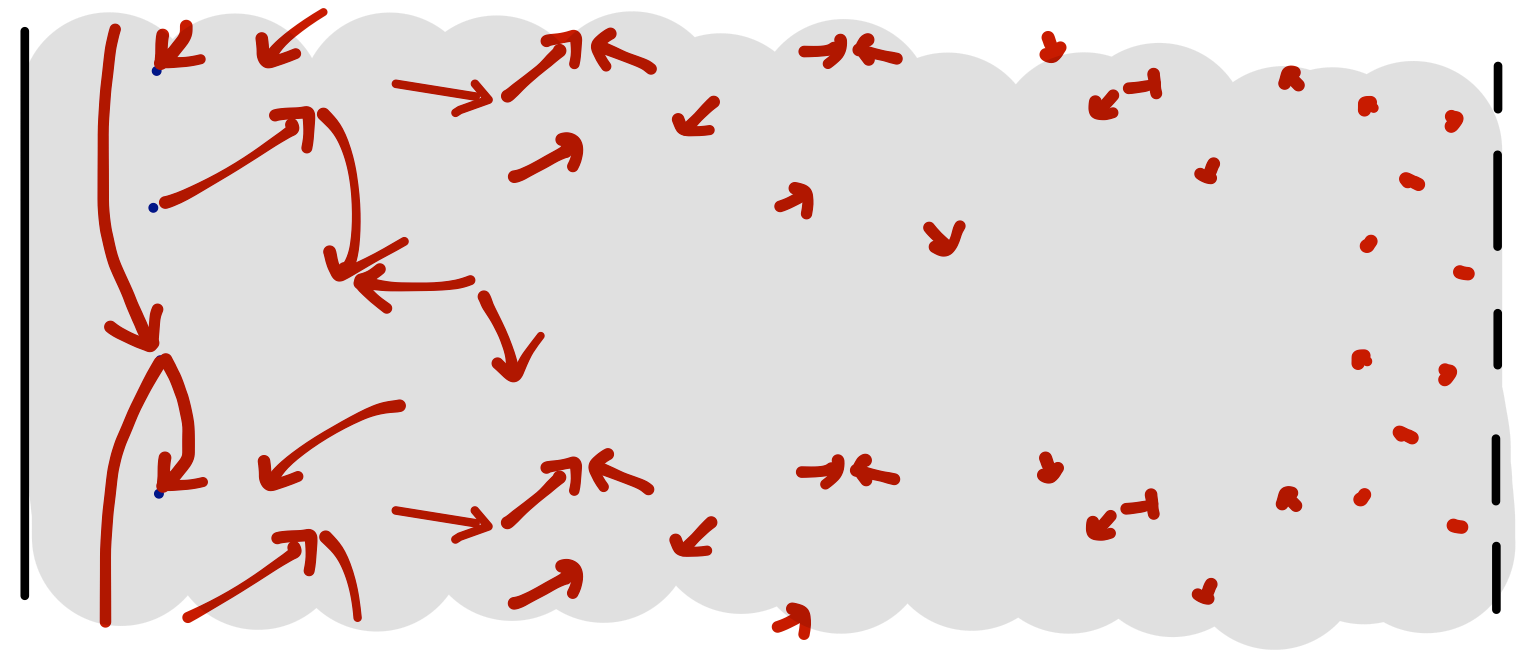}
\caption{A morphism in the obstruction category~$\CO^G( X)$.}
\end{figure}

The following result explains the choice of notation: the category $\CO^G( X)$ is the \emph{obstruction category}.
\begin{theorem} \label{obstruction-category}
The assembly map
\[
EG(\CF)_+\sma_{\Or G}\K(R[\oid{G}{-}]) \TO \K (  R[G] )
\]
is a $\pi_{\ast}$-isomorphism if and only if $K_{\ast} ( \CO^G ( EG( \CF ) ) ) = 0$.
\end{theorem}

\begin{proof}
The map $EG( \CF ) \TO \pt$ and the homotopy fibration sequence~\eqref{eq:TOD} induce the following commutative diagram with exact rows.
\[
\begin{tikzcd}[column sep=small]
\dotsb
\arrow[r]
&
K_n( \CO^G( EG(\CF)) )
\arrow[r]
\arrow[d]
&
K_n( \CD^G(EG(\CF) ) )
\arrow[r, "\ts\four"]
\arrow[d,"\ts\two"']
&
K_{n-1}( \CT^G(EG(\CF) ) )
\arrow[r]
\arrow[d, "\ts\one", "\cong"']
&
\dotsb
\\
\dotsb
\arrow[r]
&
K_n( \CO^G( \pt ) ) = 0
\arrow[r]
&
K_n( \CD^G( \pt ) )
\arrow[r, "\ts\three"', "\cong"]
&
K_{n-1}( \CT^G( \pt ) )
\arrow[r]
&
\dotsb
\end{tikzcd}
\]
The map $\one$ is an isomorphism, because source and target are both isomorphic to $K_{n-1}(R[G])$ via the forgetful map~\eqref{forget-iso}.
Using the shift map $[1, \infty) \TO{} [1, \infty)$, $x \longmapsto x +1$, it is not difficult to prove that $\CO^G( \pt )$ admits an Eilenberg swindle, and so $K_*( \CO^G( \pt ) ) = 0$. Therefore also the map~$\three$  is an isomorphism. Since the map $\two$ is identified with the assembly map, the result follows.
\end{proof}

\begin{remark}[Proof of Theorem~\ref{thm:small-in-image}]
\label{rem:small-in-image}
Consider the ladder diagram in the previous proof, but replace $EG( \CF )$ with a simplicial complex~$Z$. The maps $\one$ and~$\three$ are still isomorphisms. The maps $\two$ and~$\four$ for $n=2$ are both models
for the assembly map~$\asbl_Z$ in Theorem~\ref{thm:small-in-image}. Exactness implies that $[\alpha] \in K_1( \CT^G ( Z ))$ is in the image of the assembly map if it maps to $0\in K_1( \CO^G( Z ) )$.
The statement of Theorem~\ref{thm:small-in-image} is now a special case of~\cite[Theorem~5.3(i)]{BL-Borel}.
\end{remark}

With some additional work, the program carried out above to decompose an arbitrary $K_1$-element into summands with sufficiently small representatives can be generalized  to show that the $K$-theory of the obstruction category in Theorem~\ref{obstruction-category} vanishes.
This leads to the following theorem, which is the main result of~\cite{BL-FH}.

\begin{theorem}[Farrell-Hsiang Criterion]
\label{FH-criterion}
\index{theorem!Farrell-Hsiang Criterion}
Let $\CF$ be a family of subgroups of~$G$.
Fix a word metric on~$G$.
Assume that there exists an~$N>0$ such that for any arbitrarily large~$D>0$ there exists a surjective homomorphism $\MOR{\pr_D}{G}{F}$ to a finite group~$F$ with the following property.
For each
\[
H\in\CH_{\pr_D}=\SET*{\pr_D^{-1}(E)}{\text{hyperelementary $E\le F$}}
\]
there exist:
\begin{enumerate}[label=(\roman*)]
\item
an $H$-simplicial complex~$Z_H$ of dimension at most~$N$ and whose isotropy groups are all contained in~$\CF$;
\item
a map
\(
\MOR{f_H}{\res_H G}{Z_H}
\)
that is $H$-equivariant and $D$-contracting, i.e.,
\(
d^1(f_H(g),f_H(g'))
\leq
\tfrac{1}{D}
\,
d^G(g,g')
\)
for all $g,g'\in G$.
\end{enumerate}
Then the assembly map
\[
EG(\CF)_+\sma_{\Or G}\K(R[\oid{G}{-}])
\TO
\K(R[G])
\]
is a $\pi_*$-isomorphism.
\end{theorem}


\section{Trace methods}
\label{sec:trace}

Trace maps are maps from algebraic $K$-theory to other theories like Hochschild homology, topological Hochschild homology, and their variants, which are usually easier to compute than $K$-theory.
These trace maps have been used successfully to prove injectivity results about assembly maps in algebraic $K$-theory.
In fact, the most sophisticated trace invariant, topological cyclic homology, was invented by B\"okstedt, Hsiang, and Madsen specifically to attack the rational injectivity of the classical assembly map for~$\K(\IZ[G])$, as explained in Subsection~\ref{subsec:BHM} below.
In joint work with L\"uck and Rognes, we applied similar techniques to the Farrell-Jones assembly map, and in particular we obtained the following partial verification of Conjecture~\ref{conj:Wh-inj}; see~\cite[Theorem~1.1]{kc}. 

\begin{theorem}
\label{thm:Wh}
Assume that, for every finite cyclic subgroup~$C$ of a group~$G$, the first and second integral group homology $H_1(BZ_GC;\IZ)$ and $H_2(BZ_GC;\IZ)$ of the centralizer~$Z_G C$ of~$C$ in~$G$ are finitely generated abelian groups.
Then $G$ satisfies Conjecture~\ref{conj:Wh-inj}, i.e., the map
\begin{equation}
\colim_{H\in\obj\Sub G(\Fin)}
\wh(H)\tensor_\IZ\IQ
\TO
\wh(G)\tensor_\IZ\IQ
\end{equation}
is injective.
\end{theorem}

In this section we want to explain the ideas and the structure of the proofs of \BHM’s Theorem~\ref{thm:BHM} and its generalization, suppressing some of the technical details.
We first consider a $K_0$-analog of Theorem~\ref{thm:Wh} and explain in full detail its proof, which is an illuminating example of the trace methods.


\subsection{A warm-up example}
\label{subsec:warm-up}

\begin{proposition}
\label{prop:K_0-k-injective}
Let $\Bbbk$ be any field of characteristic zero.
Then for any group~$G$ the map
\[
\colim_{H\in\obj\Sub G(\Fin)}
K_0(\Bbbk[H])\tensor_\IZ \IQ
\TO
K_0(\Bbbk[G])\tensor_\IZ \IQ
\]
is injective.
\end{proposition}

This is closely related to Conjecture~\ref{conj:FJ-K_0-regular} for $R = \Bbbk$, but observe that, even though $K_0(\Bbbk[H])$ is a finitely generated free abelian group for each finite group $H$, the colimit in the source of the map in Conjecture~\ref{conj:FJ-K_0-regular} may contain torsion~\cite{Kropholler-Moselle}.
Therefore Proposition~\ref{prop:K_0-k-injective} does not imply the injectivity of the map in Conjecture~\ref{conj:FJ-K_0-regular}.

The key ingredient in the proof of Proposition~\ref{prop:K_0-k-injective} is the trace map
\[
\MOR{\tr}{K_0(R)}{R/[R,R]}
\,,
\]
where $[R,R]$ denotes the subgroup of the additive group of~$R$ generated by commutators. The trace map is defined as follows.
The projection $R\TO R/[R,R]$ extends to a map
\[
\MOR{\tr}{M_n(R)}{R/[R,R]}
\,,\qquad
a=(a_{ji})\longmapsto\tr(A)=\sum_{i=1}^n[a_{ii}]
\,,
\]
which is easily seen to be the universal additive map out of $M_n(R)$ with the trace property: $\tr(ab) = \tr(ba)$.
If $p$ is an idempotent matrix in $M_n(R)$, then $\tr(p)$ only depends on the isomorphism class of the projective $R$-module $R^n p$.
Since the trace sends the block sum of matrices to the sum of the traces, it induces a group homomorphism
\begin{equation}
\label{eq:tr0}
\MOR{\tr}{K_0(R)}{R/[R,R]}
\,,\qquad
[(p_{ji})]\longmapsto\sum_i[p_{ii}]
\,.
\end{equation}

Now consider the case of group algebras.
We denote by~$\conj G$ the set of conjugacy classes of elements of~$G$.
The map \(R[G]\TO R[\conj G]\) induced by the projection sends $\bigl[ R[G] , R[G] \bigr]$ to zero, and it induces an isomorphism
\[
R[G] / \bigl[ R[G] , R[G] \bigr] \cong R[\conj G]
\,.
\]
The composition of the trace map~$\tr$ from~\eqref{eq:tr0} with this isomorphism gives a map
\[
\MOR{\tr}{K_0(R[G])}{R[\conj G]}
\,,
\]
which is known as the Hattori-Stallings rank\index{Hattori-Stallings rank}\index{trace map!Hattori-Stallings rank}.
In the special case of group algebras of finite groups with coefficients in fields of characteristic zero we have the following result.

\pagebreak

\begin{lemma}
\label{lem:characters}
Suppose that the group~$G$ is finite and that $R=\Bbbk$ is a field of characteristic zero.
Let $R_\Bbbk(G)$ be the representation ring of~$G$ over~$\Bbbk$, and consider the map
\[
\MOR{\chi}{R_\Bbbk(G)}{\Bbbk[\conj G]}
\,,\qquad
\rho\longmapsto\bigl(\chi_\rho\colon g\mapsto\tr_\Bbbk(\rho(g))\bigr)
\]
that sends each representation to its character.
Then there is a commutative diagram
\[
\begin{tikzcd}
K_0 ( \Bbbk[G] )
\arrow[r, "\tr"]
\arrow[d, "\cong"']
&
\Bbbk[\conj G]
\arrow[d, "\cong"]
&
{[g]}
\arrow[mapsto, d]
\\
R_\Bbbk(G)
\arrow[r, "\chi"']
&
\Bbbk[\conj G]
&
\#\bigl(Z_G\langle g \rangle\bigr)[g^{-1}]
\end{tikzcd}
\]
whose vertical maps are isomorphisms.
\end{lemma}

In other words, the Hattori-Stallings rank can be identified up to isomorphism with the character map~$\chi$.
Notice, though, that unlike~$\chi$ the Hattori-Stallings rank is natural in~$G$.

\begin{proof}[Proof of Lemma~\ref{lem:characters}]
Since $G$ is finite and $\Bbbk$ has characteristic zero, a finitely generated projective $\Bbbk[G]$-module $V$ is the same as a finite-dimensional $\Bbbk$-vector space $V$ equipped with a linear $G$-action
$\MOR{\rho}{G}{GL(V)}$.
This explains the left vertical isomorphism in the diagram above.
It is well known that every irreducible representation is contained as a direct summand in the regular representation $\Bbbk [G]$. Therefore we can assume that the idempotent $p=p^2 = \sum_{k \in G} p_k k$ lies in $\Bbbk [G]$. Let $\langle - , - \rangle$ be the $\Bbbk$-bilinear form on $\Bbbk [ G] $ that is determined on group elements by $\langle g , h \rangle = \delta_{gh}$. Then
\begin{multline*}
\chi_{\rho} ( g )
=  \mathit{tr}_{\Bbbk} (\Bbbk[G]p \to \Bbbk[G]p, \  x \mapsto gx )
=  \mathit{tr}_{\Bbbk} ( \Bbbk[G] \to \Bbbk[G], \  x \mapsto gxp)
=
\\
=  \sum_{h \in G} \langle h , ghp \rangle
=  \sum_{h \in G}  \sum_{k \in G}  p_k \langle h , ghk \rangle
=  \sum_{h \in G}  p_{h^{-1} g^{-1} h}
=  \sum_{x \in [g^{-1}]}  \#\bigl( Z_G \langle g^{-1} \rangle\bigr) p_x
\,.
\end{multline*}
For the last equality observe that the stabilizer of $g \in G$ under the action of $G$ on itself via conjugation is the centralizer $Z_G \langle g \rangle$. For the Hattori-Stallings rank we have
$\tr (p) ( [g] ) = \sum_{x \in [g]} p_x$.
\end{proof}

We are now ready to prove Proposition~\ref{prop:K_0-k-injective}.

\begin{proof}[Proof of Proposition~\ref{prop:K_0-k-injective}]
It suffices to prove the injectivity of the map in Proposition~\ref{prop:K_0-k-injective} with $ - \tensor_\IZ \IQ$ replaced by $ - \tensor_{\IZ} \Bbbk$.
We explain the proof in the case $\Bbbk=\IC$.
Consider the following commutative diagram.
\[
\begin{tikzcd}[column sep=tiny]
\ds\colim_{H\in\obj\Sub G(\Fin)}
K_0(\IC [H])\tensor_\IZ\IC
\arrow[rr]
\arrow[d, "\cong" description, "\ts\one\ "', pos=.36, shift left=2.6em, shorten <=-1ex]
&
&
\ds
K_0(\IC [G])\tensor_\IZ\IC
\arrow[d]
\\
\ds\colim_{H\in\obj\Sub G(\Fin)}
\mathrlap{\IC [\conj H]}
\hphantom{K_0(\IC [H])\tensor_\IZ\IC }
\arrow[rr, shorten <=-1.8em]
\arrow[dr, "\cong" description, "\ts\two\:"', pos=.28, shorten <=-1.8em]
&
&
\IC [\conj G]
\\
&
\ds\IC \Bigl[\,
\colim_{H\in\obj\Sub G(\Fin)}
\conj H
\Bigr]
\arrow[ur, hook, "\ts\three"']
\end{tikzcd}
\]
The vertical maps are induced by the $\IC$-linear extension of the Hattori-Stallings rank.
For each finite group~$H$ this extension is an isomorphism by Lemma~\ref{lem:characters} and~\cite[Corollary 1 in \S 12.4]{Serre}, and so the map~$\one$ is an isomorphism.
The map~$\two$ is an isomorphism because the functor~$\IC [-]$ is left adjoint and hence preserves colimits.
Since conjugation with elements in~$G$ represents morphisms in $\Sub G ( \Fin )$, the map~$\three$ is easily seen to be injective already before applying $\IC[-]$.

The proof for an arbitrary field $\Bbbk$ of characteristic zero is completely analogous, but the set $\conj G$ needs to be replaced by the set $\conj_{\Bbbk} G$ of $\Bbbk$-conjugacy classes, a certain quotient of $\conj G$.
\end{proof}

Notice that for each finite group~$H$ the Hattori-Stallings rank itself (before $\Bbbk$-linear extension) is always injective.
But we cannot leverage this fact to prove integral injectivity results  because colimits need not preserve injectivity.


\subsection{\BHM's Theorem}
\label{subsec:BHM}

The map~$\tr$ in~\eqref{eq:tr0} is just the first (or rather the zeroth) and the easiest {trace invariant} of the algebraic $K$-theory of~$R$.
We now briefly overview how it can be generalized, starting with the Dennis trace with values in Hochschild homology.

Consider the simplicial abelian group
\begin{equation}
\label{eq:HH}
\begin{tikzcd}
\arrow[r, phantom, "\dotsb"]
&
R\tensor R\tensor R
\arrow[r, shift left=2]
\arrow[r, leftarrow, shift left=1]
\arrow[r]
\arrow[r, leftarrow, shift right=1]
\arrow[r, shift right=2]
&
R\tensor R
\arrow[r, shift left]
\arrow[r, leftarrow]
\arrow[r, shift right]
&
R
\end{tikzcd}
\end{equation}
whose face maps are
\[
d_i(r_0\tensor\dotsb\tensor r_n)=\begin{cases}
r_0\tensor\dotsb\tensor r_ir_{i+1}\tensor\dotsb\tensor r_n
&\text{if $i<n$;}
\\
r_nr_0\tensor r_1\tensor\dotsb\tensor r_{n-1}
&\text{if $i=n$.}
\end{cases}
\]
The geometric realization of the simplicial abelian group~\eqref{eq:HH} is the zeroth space of an $\Omega$-spectrum denoted~$\HH(R)=\HH(R\,|\,\IZ)$, whose homotopy groups
\[
\hh_*(R)=\pi_*\HH(R)
\quad
\text{are the \emph{Hochschild homology}\index{Hochschild homology} groups of~$R$.}
\]
In particular, we see that $\hh_0(R)$ is the cokernel of the map $r\otimes s\longmapsto rs-sr$, and hence
\[
\hh_0(R)\cong R/[R,R]
\,.
\]
The trace map~$\MOR{\tr}{K_0(R)}{\hh_0(R)}$ in~\eqref{eq:tr0} lifts to a map of spectra
\[
\MOR{\trd}{\K^{\ge0}(R)}{\HH(R)}
\]
called the \emph{Dennis trace}, such that $\pi_0\trd=\tr$.
We use $\K^{\ge0}$ to denote connective algebraic $K$-theory, the $(-1)$-connected cover of the functor $\K$ we used throughout.

Following ideas of Goodwillie and Waldhausen, Bökstedt~\cite{B-THH} introduced a far-reaching generalization of~$\HH(R)$, called \emph{topological Hochschild homology} and denoted~$\THH(R)$.
We omit the technical details of the definitions, and we rather explain the underlying ideas and structures.

The key idea in the definition of topological Hochschild homology\index{topological Hochschild homology} is to pass from the ring~$R$ to its Eilenberg-Mac Lane ring spectrum~$\IH R$, and to replace the tensor products (over the initial ring~$\IZ$) with smash products (over the initial ring spectrum~$\IS$).
In order to make this precise, one needs to work within a symmetric monoidal model category of spectra (e.g., symmetric spectra), or with ad hoc point-set level constructions (as B\"okstedt did, long before symmetric spectra and the like were discovered).
Once these technical difficulties are overcome, one obtains a simplicial spectrum
\pagebreak
\begin{equation}
\label{eq:THH}
\begin{tikzcd}
\arrow[r, phantom, "\dotsb"]
&
\IH R\sma\IH R\sma\IH R
\arrow[r, shift left=2]
\arrow[r, leftarrow, shift left=1]
\arrow[r]
\arrow[r, leftarrow, shift right=1]
\arrow[r, shift right=2]
&
\IH R\sma\IH R
\arrow[r, shift left]
\arrow[r, leftarrow]
\arrow[r, shift right]
&
\IH R
\,,
\end{tikzcd}
\end{equation}
whose geometric realization is~$\THH(R)=\HH(\IH R\,|\,\IS)$.
Notice that of course this definition applies not only to Eilenberg-Mac Lane ring spectra~$\IH R$ but to arbitrary ring spectra~$\IA$.

Bökstedt also lifted the Dennis trace to topological Hochschild homology for any connective ring spectrum~$\IA$:
\[
\begin{tikzcd}[column sep=large]
&
\THH(\IA)
\arrow[d]
\\
\K^{\ge0}(\IA)
\arrow[r,                 "\trd"']
\arrow[ru,  bend left=20, "\trb" pos=.66]
&
\HH(\pi_0\IA)
\mathrlap{\,.}
\end{tikzcd}
\]

Cyclic permutation of the tensor factors in~\eqref{eq:HH} or smash factors in~\eqref{eq:THH} makes those simplicial objects into \emph{cyclic} objects, thus inducing a natural $S^1$-action on their geometric realizations; see for example \cite[Section~3]{Jones} 
and \cite{Drinfeld}.
B\"okstedt, Hsiang, and Madsen~\cite{BHM} discovered that topological Hochschild homology has even more structure, which Hochschild homology lacks.
Fix a prime~$p$.
As $n$ varies, the fixed points of the induced $\Cp{n}$-actions are related by maps
\begin{equation}
\label{eq:RF}
\begin{tikzcd}
\ds\THH(\IA)^{\Cp{n}}
\arrow[r, shift left,  "R"]
\arrow[r, shift right, "F"']
&
\ds\THH(\IA)^{\Cp{n-1}}
\,,
\end{tikzcd}
\end{equation}
called Restriction and Frobenius.
The map~$F$ is simply the inclusion of fixed points, whereas the definition of the map~$R$ is much more delicate and specific to the construction of~$\THH$.
The homotopy equalizer of~\eqref{eq:RF} is denoted~$\TC^{n+1}(\IA;p)$.
One important property of the maps $R$ and~$F$ is that they commute, and therefore they induce a map
\(
\TC^{n+1}(\IA;p)\TO\TC^n(\IA;p)
\,.
\)
The \emph{topological cyclic homology}\index{topological cyclic homology} of~$\IA$ at the prime~$p$ is then defined as the homotopy limit
\[
\TC(\IA;p)=\holim_n\TC^n(\IA;p)
\,.
\]

B\"okstedt, Hsiang, and Madsen lifted the B\"okstedt trace to topological cyclic homology, thus obtaining the following commutative diagram for any connective ring spectrum~$\IA$:
\begin{equation}
\label{eq:traces}
\begin{tikzcd}[column sep=large]
&
\TC(\IA;p)
\arrow[d]
\\
&
\THH(\IA)
\arrow[d]
\\
\K^{\ge0}(\IA)
\arrow[r,                 "\trd"']
\arrow[ru,  bend left=20, "\trb" pos=.66]
\arrow[ruu, bend left=36, "\trc"]
&
\HH(\pi_0\IA)
\mathrlap{\,.}
\end{tikzcd}
\end{equation}
The map~$\trc$ is called the \emph{cyclotomic trace map}\index{cyclotomic trace map}\index{trace map!cyclotomic}.

They then used this technology to prove the following striking theorem, which is often referred to as the algebraic $K$-theory Novikov Conjecture; see~\cite[Theorem~9.13]{BHM} 
and~\cite[Theorem~4.5.4]{Madsen}. 

\pagebreak

\begin{theorem}[\BHM]
\label{thm:BHM}
\index{theorem!Bökstedt-Hsiang-Madsen}
Let $G$ be a group.
Assume that the following condition holds.
\begin{enumerate}[label=\assum{\Alph*}{1}, leftmargin=*]
\item
\label{BHM-A}
For every $s\geq1$, the integral group homology $H_s(BG;\IZ)$ is a finitely generated abelian group.
\end{enumerate}
Then the classical assembly map
\[
\MOR{\asbl_1}{BG_+\sma\K(\IZ)}{\K(\IZ[G])}
\]
is $\pi_*^\IQ$-injective, i.e., $\pi_n(\asbl_1)\tensor_\IZ\IQ$ is injective for all~$n\in\IZ$.
\end{theorem}

We now explain the structure of the proof of Theorem~\ref{thm:BHM}, following the approach of~\cite{kc}.
As mentioned above, the idea is to use the cyclotomic trace map.
However, it is not enough to work with topological cyclic homology, and one needs a variant of it that we proceed to explain.
Instead of taking the homotopy equalizer of~$R$ and~$F$ in~\eqref{eq:RF}, we may consider just the homotopy fiber of~$R$ and define
\[
\C^{n+1}(\IA;p)=\hofib\Bigl(
\THH(\IA)^{\Cp{n}}
\TO[R]
\THH(\IA)^{\Cp{n-1}}
\Bigr)
\,.
\]
The map~$F$ induces a map
\(
\C^{n+1}(\IA;p)\TO\C^n(\IA;p)
\,,
\)
and we define
\[
\C(\IA;p)=\holim_n\C^n(\IA;p)
\,.
\]
A fundamental property, also established in~\cite{BHM}, is that $\C^{n+1}(\IA;p)$ can be identified with $\THH(A)_{h\Cp{n}}$, up to a zigzag of $\pi_*$-isomorphisms.
In~\cite[Section~8]{kc} we provided a \emph{natural} zigzag of $\pi_*$-isomorphisms between $\THH(A)_{h\Cp{n}}$ and~$\C^{n+1}(\IA;p)$, natural even before passing to the stable homotopy category of spectra.
The key tool here is the natural Adams isomorphism for equivariant orthogonal spectra developed in~\cite{RV}.

In the special case when $\IA=\IS[G]$ is a spherical group ring, then the maps~$R$ split, and these splittings can be used to construct a map
\begin{equation}
\label{eq:TC->C}
\TC(\IS[G];p)\TO\C(\IS[G];p)
\,.
\end{equation}
The crucial advantage of using~$\C$ instead of~$\TC$ is that more general rational injectivity statements can be proved for the assembly maps for~$\C$; compare Remark~\ref{rem:C-vs-TC} below.

In order to prove Theorem~\ref{thm:BHM} one studies the following commutative diagram.
\begin{equation}
\label{eq:BHM}
\begin{tikzcd}
\ds BG_+\sma\K(\IZ)
\arrow[rr, "\asbl_1"]
&
&
\K(\IZ[G])
\\
\ds BG_+\sma\K^{\ge0}(\IZ)
\arrow[rr]
\arrow[u, "\ts\one"]
&
&
\K^{\ge0}(\IZ[G])
\arrow[u, "\ts\bone"']
\\
\ds BG_+\sma\K^{\ge0}(\IS)
\arrow[rr]
\arrow[u, "\ts\two"]
\arrow[d, "\ts\three"']
&
&
\K^{\ge0}(\IS[G])
\arrow[u, "\ts\btwo"']
\arrow[d, "\ts\bthree"]
\\
\ds BG_+\sma\TC(\IS;p)
\arrow[rr]
\arrow[d, "\ts\four"']
&
&
\TC(\IS[G];p)
\arrow[d, "\ts\bfour"]
\\
\ds BG_+\sma\bigl(\THH(\IS)\times\C(\IS;p)\bigr)
\arrow[rr, "\ts\five"']
&
&
\THH(\IS[G])\times\C(\IS[G];p)
\end{tikzcd}
\end{equation}
The horizontal maps are all classical assembly maps, and we want to prove that the one at the top of the diagram is $\pi_*^\IQ$-injective.
The maps $\one$ and~$\bone$ are induced by the natural maps from connective to non-connective algebraic $K$-theory.
Since $\IZ$ is regular, $\one$ is a $\pi_*$-isomorphism.
The maps $\two$ and~$\btwo$ come from the linearization (or Hurewicz) map~$\IS\TO\IZ$, and they are both $\pi_*^\IQ$-isomorphisms by a result of Waldhausen~\cite[Proposition~2.2]{Waldhausen-A1}. 
The maps $\three$ and~$\bthree$ are given by the cyclotomic trace map, and $\four$ and~$\bfour$ by the natural maps in \eqref{eq:traces} and~\eqref{eq:TC->C}.

So, in order to prove that the top horizontal map in diagram~\eqref{eq:BHM} is $\pi_*^\IQ$-injective, it is enough to show that:
\begin{enumerate}[label=(\alph*)]
\item
\label{i:splitting}
The assembly map~$\five$ is $\pi_*^\IQ$-injective.
\item
\label{i:detection}
The composition $\four\circ\three$ is $\pi_*^\IQ$-injective.
\end{enumerate}
The assumption~\ref{BHM-A} is then shown to imply~\ref{i:splitting}, and in fact not just for~$\IS$ but for arbitrary connective ring spectra~$\IA$.
This is the special case~$\CF=1$ of Theorems~\ref{thm:thh-iso} and~\ref{thm:C-inj} below.
The difficult part in proving~\ref{i:detection} is the analysis of the map~$\three$.
The Atiyah-Hirzebruch spectral sequences collapse rationally, and therefore it is enough to study the rational injectivity of $\MOR{\trc}{\K^{\ge0}(\IS)}{\TC(\IS;p)}$.
To this end, consider the following commutative diagram.
\begin{equation}
\label{eq:detection}
\begin{tikzcd}[row sep=tiny]
&
&
&
\ds\K^{\ge0}(\IZ_p)
\arrow[rd, "\ts\seven"]
\\
\ds\K^{\ge0}(\IS)
\arrow[rr, "\ts\two"]
\arrow[ddd, "\trc\,"']
&
&
\ds\K^{\ge0}(\IZ)
\arrow[ru, "\ts\six"]
\arrow[rd, "\ts\eight"']
&
&
\ds\K^{\ge0}(\IZ_p)^\compl_p
\arrow[ddd, "\,\trc^\compl_p"]
\\
&
&
&
\ds\K^{\ge0}(\IZ)^\compl_p
\arrow[ru, "\ts\nine"']
\\
\\
\ds\TC(\IS;p)
\arrow[rrrr]
&
&
&
&
\ds\TC(\IZ_p;p)^\compl_p
\end{tikzcd}
\end{equation}
Here $( - )^\compl_p$ denotes the $p$-completion of spectra and $\IZ_p$ are the $p$-adic numbers.
The map $\trc^\compl_p$ is a $\pi_n$-isomorphism for each~$n\ge0$ by a result of Hesselholt and Madsen~\cite[Theorem~D]{HM}. 
We already  mentioned above that $\two$ is a $\pi_*^{\IQ}$-isomorphism.

It remains to discuss the diamond. Since the groups $K_n( \IZ)$ are known to be finitely generated, $\eight$ is $\pi_*^{\IQ}$-injective.
The question whether  $\nine$ is $\pi_*^{\IQ}$-injective is open in general. It can be reformulated in terms of similar maps in {\'e}tale $K$-theory, {\'e}tale cohomology, or Galois cohomology, as surveyed in~\cite[Section~18]{kc}. Luckily the equivalent conjecture in Galois cohomology is known to be true if $p$ is a regular prime by results in \cite{Schneider}; see~\cite[Proposition~2.9]{kc}.
Recall that a  prime~$p$ is regular if it does not divide the order of the ideal class group of~$\IQ( \zeta_p)$. Since regular primes exist we obtain the following statement and we are done.
\begin{enumerate}[label=\assum{\Alph*}{1}, leftmargin=*, start=2]
\item
There exists a prime $p$ such that $\nine \circ \eight$ is $\pi_*^{\IQ}$-injective.
\end{enumerate}
We remark that little is known about the rationalized homotopy groups of~$\K^{\ge0}( \IZ_p )$ without $p$-completion; compare \cite[Warning 60]{Weibel05}. 

This concludes our explanation of the proof of Theorem~\ref{thm:BHM}.


\subsection{Generalizations}

The following result generalizes Theorem~\ref{thm:BHM} from the classical to the Farrell-Jones assembly map, and is a special case of \cite[Main Technical Theorem~1.16]{kc}. 

\pagebreak

\begin{theorem}
\label{thm:lrrv-main}
\index{theorem!Bökstedt-Hsiang-Madsen!generalization}
Let $G$ be a group and let~$\CF\subseteq\FCyc$ be a family of finite cyclic subgroups of~$G$.
Assume that the following two conditions hold.
\begin{enumerate}[label=\assum{\Alph*}{\CF}, leftmargin=*]
\item
\label{our-A}
For every $C\in\CF$ and every~$s\geq1$, the integral group homology $H_s(BZ_GC;\IZ)$ of the centralizer of~$C$ in~$G$ is a finitely generated abelian group.
\item
\label{our-B}
For every $C\in\CF$ and every~$t\geq0$, the natural homomorphism
\[
K_t(\IZ[\zeta_c])\tensor_\IZ\IQ
\TO
\;\smallprod_{\mathclap{p\textup{ prime}}}\;
K_t\Bigl(\IZ_p\tensor_\IZ\IZ[\zeta_c];\IZ_p\Bigr)\tensor_\IZ\IQ
\]
is injective, where $c$ is the order of~$C$, $\zeta_c$ is any primitive $c$-th root of unity, and $K_t(R;\IZ_p)=\pi_t(\K(R)^\compl_p)$.
\end{enumerate}
Then the assembly map
\[
\MOR{\asbl_\CF}{EG(\CF)_+\sma_{\Or G}\K^{\ge0}(\IZ[\oid{G}{-}])}{\K^{\ge0}(\IZ[G])}
\]
is $\pi_*^\IQ$-injective.
\end{theorem}

\begin{remark}
\label{rem:lrrv-main}
Several comments are in order.
\begin{enumerate}[label=(\roman*)]
\item
When $\CF=1$ is the trivial family, Theorems~\ref{thm:BHM} and~\ref{thm:lrrv-main} coincide.
This is because assumption~\assum{A}{1} of Theorem~\ref{thm:lrrv-main} is literally the same as assumption~\ref{BHM-A} of Theorem~\ref{thm:BHM}, and assumption~\assum{B}{1} follows at once from the corresponding true statement explained at the end of the previous subsection.

\item
When $\CF=\FCyc$, then the rationalized assembly map for \emph{connective} algebraic $K$-theory studied in Theorem~\ref{thm:lrrv-main} can be rewritten as in Conjecture~\ref{conj:FJ-rationalized}, because the isomorphisms \eqref{eq:rel-asbl-rationally} and~\eqref{eq:Lueck-Grunewald} hold for both connective and non-connective algebraic $K$-theory.
The only difference is that the summands indexed by~$t=-1$ in the source of the map in Conjecture~\ref{conj:FJ-rationalized} are now missing.
Notice that the negative $K$-groups $K_t(\IZ[C])$ are known to vanish for any~$t<-1$ if $C$ is finite or even virtually cyclic~\cite{FJ-vcyc}.

\item
\label{i:on-A}
As noted above, assumption \ref{our-A} implies and is the obvious generalization of assumption~\ref{BHM-A}.
For any $\CF\subseteq\FCyc$, assumption~\ref{our-A} is satisfied if there is a universal space~$EG(\Fin)$ of finite type, i.e., whose skeleta are all cocompact.
Hyperbolic groups, finite-dimensional CAT(0)-groups, cocompact lattices in virtually connected Lie groups, arithmetic groups in semisimple connected linear $\IQ$-algebraic groups, mapping class groups, and outer automorphism groups of free groups are all examples of groups that even have a finite-dimensional and cocompact $EG(\Fin)$.
Among these groups, outer automorphism groups of free groups do not appear in Theorem~\ref{thm:state}, and for them Theorem~\ref{thm:lrrv-main} gives the first result about the Farrell-Jones Conjecture.
An interesting example of a group that satisfies \assum{A}{\FCyc} without having an~$EG(\Fin)$ of finite type is given by Thompson’s group~$T$ of orientation preserving, piecewise linear, dyadic homeomorphisms of the circle; see~\cite{GV}.

\item
\label{i:our-B}
Conjecturally assumption~\ref{our-B} of Theorem~\ref{thm:lrrv-main} is always satisfied; in fact, it is implied by a weak version of the Leopoldt-Schneider Conjecture for cyclotomic fields, as explained carefully in \cite[Sections~2 and~18]{kc}.
When $t=0$ or $t=1$, i.e., for $K_0$ and $K_1$, the map in~\ref{our-B} is injective for arbitrary $c$ by direct computation; compare \cite[Proposition~2.4]{kc}.
For any fixed~$c$ it is known that injectivity may fail for at most finitely many values of~$t$. 
These two facts allow to deduce Theorems~\ref{thm:Wh} and~\ref{thm:inj-results}\ref{i:eventual-inj} from Theorem~\ref{thm:lrrv-main}, or rather from its more general version in~\cite[Main Technical Theorem~1.16]{kc}, as explained in loc.\ cit., Section~17 and page~1015.
Notice that, on the other hand, Theorem~\ref{thm:BHM} cannot be used to deduce information about the Whitehead group~$\wh(G)$, which is the cokernel of the map induced on~$\pi_1$ by the classical assembly map~$\asbl_1$.
\end{enumerate}
\end{remark}

The proof of Theorem~\ref{thm:lrrv-main} follows the same strategy as the proof of Theorem~\ref{thm:BHM} outlined above.
We consider the analog of diagram~\eqref{eq:BHM} for the generalized assembly map~$\asbl_\CF$; compare~\cite[``main diagram''~(3.1)]{kc}. 
The key results about assembly maps are summarized in the following two theorems~\cite[Theorem~1.19, parts~(i) and~(ii)]{kc}.
We point out that all the following results hold for arbitrary connective ring spectra~$\IA$.

\begin{theorem}
\label{thm:thh-iso}
\index{topological Hochschild homology}
For any group~$G$ and for any family~$\CF$ of subgroups of~$G$, the assembly map
\[
\MOR{\asbl_\CF}{EG(\CF)_+\sma_{\Or G}\THH(\IA[\oid{G}{-}])}{\THH(\IA[G])}
\]
induces split monomorphisms on~$\pi_*$, and it is a $\pi_*$-isomorphism if and only if~$\CF$ contains all cyclic subgroups of~$G$, i.e., $\CF\supseteq\Cyc$.
\end{theorem}

\begin{theorem}
\label{thm:C-inj}
Let $G$ be a group and let~$\CF\subseteq\FCyc$ be a family of finite cyclic subgroups of~$G$.
Assume that the following condition holds.
\begin{enumerate}[label=\assum{\Alph*}{\CF}, leftmargin=*]
\item
For every $C\in\CF$ and every~$s\geq1$, the integral group homology $H_s(BZ_GC;\IZ)$ of the centralizer of~$C$ in~$G$ is a finitely generated abelian group.
\end{enumerate}
Then the assembly map
\[
\MOR{\asbl_\CF}{EG(\CF)_+\sma_{\Or G}\C(\IA[\oid{G}{-};p])}{\C(\IA[G];p)}
\]
is $\pi_*^\IQ$-injective.
\end{theorem}

\begin{remark}
\label{rem:C-vs-TC}
In order to establish an analog of Theorem~\ref{thm:C-inj} for the assembly map
\begin{equation}
\label{eq:C-vs-TC}
\MOR{\asbl_\CF}{EG(\CF)_+\sma_{\Or G}\TC(\IA[\oid{G}{-};p])}{\TC(\IA[G];p)}
\end{equation}
in topological cyclic homology, we need to assume not only condition~\ref{our-A}, but also the following two conditions:
\begin{itemize}[leftmargin=2.75em]
\item[\assum{A'}{\CF}]
the family $\CF$ contains only finitely many conjugacy classes of subgroups;
\item[\assum{A''}{\CF}]
for every~$g\in G$, $\langle g\rangle\in\CF$ if and only if $\langle g^p\rangle\in\CF$.
\end{itemize}
The fact that the assembly map~\eqref{eq:C-vs-TC} is $\pi_*^\IQ$-injective under assumptions \assum{A}{\CF}, \assum{A'}{\CF}, and~\assum{A''}{\CF} is a special case of~\cite[Theorem~1.8]{tc}.
Notice the following facts.
\begin{enumerate}[label=(\roman*)]
\item
When $\CF=1$, assumption \assum{A'}{1} is vacuously true, but \assum{A''}{1} is \emph{not} satisfied if $G$ has $p$-torsion.
This is the reason why, in the proof of \BHM's Theorem~\ref{thm:BHM}, we need to work with~$\C$ and not just~$\TC$.

\item
As pointed out in Remark~\ref{rem:lrrv-main}\ref{i:on-A}, Thompson's group~$T$ satisfies~\assum{A}{\FCyc} and obviously also~\assum{A''}{\FCyc}.
However, $T$ contains finite cyclic subgroups of any given order, and therefore does not satisfy~\assum{A'}{\FCyc}.
It is an interesting open question whether the assembly map~\eqref{eq:C-vs-TC} is $\pi_*^\IQ$-injective for~$G=T$.

\item
Without homological finiteness assumptions on~$G$, the assembly map~\eqref{eq:C-vs-TC} is not rationally injective in general.
For example, if $G=\IQ$ and~$\CF=1=\FCyc$, then~\eqref{eq:C-vs-TC} is essentially trivial after applying~$\pi_*(-)\tensor_\IZ\IQ$.
This is explained in~\cite[Remark~3.7]{kc}.
Of course, the group~$G=\IQ$ does not satisfy~\ref{BHM-A}.
\end{enumerate}
\end{remark}

Finally, we mention the following two additional results about assembly maps for topological cyclic homology, which we proved in~\cite[Theorems~1.1, 1.4(ii), and~1.5]{tc}.

One should view Theorem~\ref{thm:tc-finite-groups} as a cyclic induction theorem for the topological cyclic homology of any finite group, with coefficients in any connective ring spectrum.
It allows to reduce the computation of~$\TC$ of any finite group to the case of the finite cyclic subgroups; this is carried out explicitly in~\cite[Proposition~1.2]{tc} for the basic case of the symmetric group on three elements.

Theorem~\ref{thm:tc-not-surj} studies the analog for~$\TC$ of the Farrell-Jones Conjecture~\ref{conj:FJ}.
For a large class of groups (for which Conjecture~\ref{conj:FJ} is already known; see Theorem~\ref{thm:state}), we prove that~$\asbl_\VCyc$ is injective, but surprisingly not surjective.

\begin{theorem}
\label{thm:tc-finite-groups}
\index{topological cyclic homology}
For any \emph{finite} group~$G$ the assembly map
\[
\MOR{\asbl_\Cyc}{EG(\Cyc)_+\sma_{\Or G}\TC(\IA[\oid{G}{-}];p)}{\TC(\IA[G];p)}
\]
is a $\pi_*$-isomorphism.
\end{theorem}

\begin{theorem}
\label{thm:tc-not-surj}
Assume that $G$ is either hyperbolic or virtually finitely generated abelian.
Then the assembly map
\[
\MOR{\asbl_\VCyc}{EG(\VCyc)_+\sma_{\Or G}\TC(\IA[\oid{G}{-}];p)}{\TC(\IA[G];p)}
\]
is always injective but in general not surjective on homotopy groups.
For example, it is not surjective on~$\pi_{-1}$ if $\IA=\IZ_{(p)}$ and $G$ is either finitely generated free abelian or torsion-free hyperbolic, but not cyclic.
\end{theorem}


\bigskip
\bigskip
\bigskip

\bibliographystyle{alpha}
\bibliography{survey}

\printindex

\vfill

\end{document}